\documentclass{amsart}
\title{Remarks on countable tightness}
\author{Marion Scheepers}

\usepackage{color}

\newtheorem{definition}{{\bf Definition}}
\newtheorem{theorem}{{\bf Theorem}}
\newtheorem{corollary}[theorem]{{\bf Corollary}}
\newtheorem{lemma}[theorem]{{\bf Lemma}}
\newtheorem{proposition}[theorem]{{\bf Proposition}}

\def\Proof{\textbf{Proof.  }}

\newcommand{\reals}{{\mathbb R}}
\newcommand{\naturals}{{\mathbb N}}
\newcommand{\gone}{{\sf G}_1}
\newcommand{\sone}{{\sf S}_1}
\newcommand{\sfin}{{\sf S}_{fin}}
\newcommand{\gfin}{{\sf G}_{fin}}
\newcommand{\cohen}{{\mathbb C}}
\newcommand{\forces}{\mathrel{\|}\joinrel\mathrel{-}}
\newcommand{\poset}{{\mathbb P}}
\newcommand{\open}{\mathcal{O}}

\subjclass[2000]{03E05, 03E35, 54A35, 54D65}
\keywords{Selection principle, countable strong fan tightness, indestructibly countably tight, HFD, generic left separated space, homogeneous ${\sf T}_5$ compactum, selective separability, infinite game}
\address{Department of Mathematics\\ Boise State University\\ Boise, Idaho 83725}
\email{mscheepe@boisestate.edu}

\begin{document}
\begin{abstract}
Countable tightness may be destroyed by countably closed forcing. We characterize the indestructibility of countable tightness under countably closed forcing by combinatorial statements similar to the ones Tall used to characterize indestructibility of the Lindel\"of property under countably closed forcing. We consider the behavior of countable tightness in generic extensions obtained by adding Cohen reals. We show that certain classes of well-studied topological spaces are indestructibly countably tight. Stronger versions of countable tightness, including selective versions of separability, are further explored.
\end{abstract}
\maketitle

Let $(X,\tau)$ be a topological space and let $x$ be an element of $X$. We say that $X$ is \emph{countably tight at }$x$ if there is for each set $A\subseteq X$ with $x\in\overline{A}$, a countable set $B\subseteq A$ such that $x\in \overline{B}$. If the space is countably tight at each of its elements, we say that the space has countable tightness or, equivalently, that the space is countably tight.

It is known that in generic extensions by countably closed posets a ground model space that is countably tight may fail, in the generic extension, to still be a countably tight space. In \cite{AD2} Dow gives an ingenious proof, using reflection arguments, that in the generic extension by an iteration of first the Cohen poset for adding uncountably many Cohen reals, then any countably closed poset, countably tight topological spaces from the ground model remain countably tight. A similar phenomenon regarding the preservation of the Lindel\"of property has been shown by Dow in \cite{AD1}. In \cite{MSFT} we gave an explanation for this phenomenon for Lindel\"of spaces. In Section 1 we show that for reasons very analogous to the Lindel\"of case, this preservation happens for countable tightness.

In Section 2 we more closely investigate the effect of Cohen forcing on countable tightness and establish a connection with countable strong fan tightness. In the subsequent four sections we investigate indestructibility of countable tightness in several classes of topological spaces of countable tightness.

\section{Indestructibility of countable tightness by countably closed forcing}

By analogy with the Lindel\"of case in \cite{Tall}, we say that a topological space is \emph{indestructibly countably tight} if the space is countably tight, and in any generic extension by countably closed forcing the space is still countably tight.
For convenience define, for $x\in X$ not an isolated point of $X$, $\Omega_x = \{A\subseteq X: x\in\overline{A}\setminus A\}$.
Thus, $X$ is countably tight at $x$ if each element of $\Omega_x$ has a countable subset which is in $\Omega_x$. 
From now on, assume that $x$ is an element of $X$, and that $X$ is countably tight at $x$. Following \cite{Tall} for the corresponding notion for Lindel\"of spaces, we define:
\begin{definition}\label{xtightnesstree} A set $T=\{y_f:f\in \cup_{\alpha<\omega_1}\,^{\alpha}\omega\}\subseteq X$ is an \emph{x-tightness tree} if for each $\alpha<\omega_1$ and for each $f\in\,^{\alpha}\omega$ we have $\{y_{f\cup\{(\alpha,n)\}}:\, n<\omega\} \in \Omega_x$.
\end{definition}

We also introduce the following infinite two-person game of length $\alpha$, denoted $\gone^{\alpha}(\Omega_x,\Omega_x)$: In inning $\beta<\alpha$ player ONE first selects an $O_{\beta}\in\Omega_x$, and TWO responds with an $x_{\beta}\in O_{\beta}$.
A play $O_0,\, x_0,\, O_1,\, x_1,\, \cdots,\, O_{\beta},\, x_{\beta},\, \cdots \hspace{0.1in} \beta<\alpha$
is won by Player TWO if $\{x_{\beta}:\beta<\alpha\}\in \Omega_x$; else, ONE wins. In \cite{COC3} this game was examined for the case when $\alpha=\omega$, and in \cite{Gamelength} it was investigated for additional countable lengths.

\begin{theorem}\label{indestrtightchar} For a topological space $X$ which is countably tight at the element $x\in X$, the following are equivalent:
\begin{enumerate}
  \item $X$ is indestructibly countably tight at $x$.
  \item The countable tightness of $X$ at $x$ is preserved upon forcing with ${\sf Fn}(\omega_1,\omega,\omega_1)$, the poset for adding a Cohen subset of $\omega_1$ with countable conditions.
  \item For each $x$-tightness tree $\{y_f:\, f\in \cup_{\alpha<\omega_1}\,^{\alpha}\omega\}$ the set $\{g\in \cup_{\alpha<\omega_1}\,^{\alpha}\omega: \{y_{g\lceil_{\gamma}}:\gamma<{\sf dom}(g)\}\in\Omega_x\}$ is dense in ${\sf Fn}(\omega_1,\omega,\omega_1)$.
  \item For each $x$-tightness tree $\{y_f:\, f\in \cup_{\alpha<\omega_1}\,^{\alpha}\omega\}$ and for each $g\in \cup_{\alpha<\omega_1}\,^{\alpha}\omega$ there is an $f\in\,^{\omega_1}\omega$ such that $f\lceil_{{\sf dom}(g)} = g$ and $\{y_{f\lceil_{\alpha}}: \alpha<\omega_1\} \in \Omega_x$.
  \item For each $x$-tightness tree $\{y_f:\, f\in \cup_{\alpha<\omega_1}\,^{\alpha}\omega\}$ there is an $f\in\cup_{\alpha<\omega_1}\,^{\alpha}\omega$ such that $\{y_{f\lceil_{\alpha}}: \alpha<{\sf dom}(f)\} \in \Omega_x$.
  \item ONE does not have a winning strategy in the game $\gone^{\omega_1}(\Omega_x,\Omega_x)$.
\end{enumerate}
\end{theorem}
$\Proof$ 
That $(2)\Rightarrow (1)$: We follow the argument in \cite{Shelah}, adapted to the current context. Thus, let $(\poset,<)$ be  a countably closed partially ordered set, and assume that ${\mathbf 1}_{\poset} \forces ``\check{X} \mbox{ is not countably tight at }\check{x}$." Fix a $\poset$-name $\dot{A}$ and a $p\in\poset$ such that
$p \forces ``\dot{A}\in \Omega_x \mbox{ but for each countable }C\subseteq \dot{A}, \, C\not\in\Omega_x."$
Put $p_{\emptyset} = p$ and let $\eta_{\emptyset}$ be the empty sequence of length $0$. Define
$F_{\emptyset} = \{y\in X: (\exists q\le p_{\emptyset})(q\forces ``\check{y}\in \dot{A}")\}.$
{\flushleft{{\bf Claim 1: }$F_{\emptyset} \in \Omega_x$:}}\\
Suppose that on the contrary $x$ is not in the closure of $F_{\emptyset}$. Choose a neighborhood $U$ of $x$ with $U\cap F_{\emptyset}=\emptyset$. As $p\forces``\dot{A}\in\Omega_{\check{x}}"$ we find that $p\forces ``\check{U}\cap\dot{A}\neq\emptyset"$. Choose a $q\le p$ and a $y\in X$ auch that $q\forces``\check{y}\in\check{U}\cap\dot{A}"$. Then $y$ is in $F_{\emptyset}$, and as $y$ is in $U$ we find the contradiction that $U\cap F_{\emptyset}\neq\emptyset$. This completes the proof of Claim 1.

As $X$ is countably tight, choose a countable $C_{\emptyset}\subseteq F_{\emptyset}$ with $x\in\overline{C_{\emptyset}}$. Enumerate $C_{\emptyset}$ bijectively as $(y_{n}:n<\omega).$ For each $n$ choose $p_{(n)}<p_{\emptyset}$ such that $p_{(n)}\forces ``\check{y}_n\in\dot{A}"$.

This specifies $p_{\eta}$ and $y_{\eta}$ for $\eta\in\,^1\omega$. Now let $0<\alpha<\omega_1$ be given, and assume that for each $\beta<\alpha$ and $\sigma\in\,^{\beta}\omega$ we have selected $p_{\sigma}\in\mathbb{P}$ and $y_{\sigma}\in X$ such that for all$\gamma<\beta$ we have $\{y_{\sigma\lceil_{\gamma}\frown\{(\gamma,n)\}}:n<\omega\}\in\Omega_x$ and $p_{\sigma}\forces ``y_{\sigma}\in \dot{A}"$.

Now we distinguish two cases: $\alpha$ is a limit ordinal, or $\alpha$ is a successor ordinal. 
{\flushleft{Case 1: }}$\alpha$ is a limit ordinal. Then for each $\sigma\in \,^{\alpha}\omega$ choose a $p_{\sigma}\in\poset$ such that for each $\beta<\alpha$ we have $p_{\sigma}<p_{\sigma\lceil_{\beta}}$. This is possible since $\poset$ is countably closed.

{\flushleft{Case 2: }}$\alpha$ is a successor ordinal.
Say $\alpha = \beta+1$. For each $\sigma\in\,^{\beta}\omega$ define
\[
  F_{\sigma} = \{y\in X: (\exists q\le p_{\sigma})(q\forces ``\check{y}\in \dot{A}")\}.
\]
{\flushleft{{\bf Claim 2: }$F_{\sigma} \in \Omega_x$:}}\\
The proof of Claim 2 proceeds like the proof of Claim 1.

As $X$ is countably tight, choose a countable $C_{\sigma}\subseteq F_{\sigma}$ with $x\in\overline{C_{\sigma}}$. Enumerate $C_{\sigma}$ bijectively as $(y_{\sigma\cup\{(\beta,n)\}}:n<\omega)$ and for each $n$ choose $p_{\sigma\cup\{(\beta,n)\}}<p_{\sigma}$ such that $p_{\sigma\cup\{(\beta,n)\}}\forces ``\check{y}_{\sigma\cup\{(\beta,n)\}}\in\dot{A}"$.

This defines $p_{\tau}$ for each $\tau$ in $^{\alpha}\omega$,  and when $\alpha$ is a successor ordinal this also defines each $y_{\tau}$. 

We now show that in the generic extension by ${\sf Fn}(\omega_1,\omega,\omega_1)$ $X$ fails to have countable tightness. For let $g\in\, ^{\omega_1}\omega$ be ${\sf Fn}(\omega_1,\omega,\omega_1)$-generic, and put
\[
  B = \{y_{g\lceil{\alpha}}:\alpha<\omega_1\}.
\]
{\flushleft{\bf Claim 3: }}$B\in\Omega_x$.

For consider any open neighborhood $U$ of $x$. Then $D_{U} = \{\eta\in\,\cup_{\alpha<\omega_1}\,^{\alpha}\omega: y_{\eta}\in U\}$ is a dense subset of ${\sf Fn}(\omega_1,\omega,\omega_1)$. For let any $p\in{\sf Fn}(\omega_1,\omega,\omega_1)$ be given. We may assume that ${\sf dom}(p)=\alpha<\omega_1$. Since $\{y_{p\frown\{(\alpha,n)\}}:n<\omega\}$ is a member of $\Omega_x$, we have $U\cap \{y_{p\frown\{(\alpha,n)\}}:n<\omega\}\neq \emptyset$. Choose $n<\omega$ with $y_{p\frown\{(\alpha,n)\}}\in U$. Then $q = p\frown\{(\alpha,n)\} \in D_U$ and $q<p$. Since the generic filter producing $g$ meets this dense set we find that $U\cap B\neq\emptyset$. It follows that $B$ is a member of $\Omega_x$.

{\flushleft{\bf Claim 4: }}No countable subset of $B$ is in $\Omega_x$.

For fix a $\beta<\omega_1$ and consider $B_{\beta} = \{y_{g\lceil_{\gamma}}:\gamma<\beta\}$. If it were the case that $B_{\beta}$ is an element of $\Omega_x$, then we would have $p_{\beta}\forces ``\check{B}_{\beta}\subseteq\dot{A}"$ and $p_{\beta}\forces``\check{B}_{\beta}\in\Omega_x"$. This contradicts the selection of the $\poset$-name $\dot{A}$.

{\flushleft{ That $(1)\Rightarrow (3)$: }} 
Assume that (3) fails. Fix an $x$-tightness tree $\{y_f:\, f\in \cup_{\alpha<\omega_1}\,^{\alpha}\omega\}$ that witnesses this failure. Since the set $D:=\{g\in \cup_{\alpha<\omega_1}\,^{\alpha}\omega: \{y_{g\lceil_{\gamma}}:\gamma<{\sf dom}(g)\}\in\Omega_x\}$ is not dense in ${\sf Fn}(\omega_1,\omega,\omega_1)$, fix a $p\in {\sf Fn}(\omega_1,\omega,\omega_1)$ for which there is no $g\in D$ with $g<p$. We may assume that ${\sf dom}(p)=\alpha<\omega_1$. Then for each $g\in \cup_{\alpha<\omega_1}\,^{\alpha}\omega$ with $g<p$ we have $\{y_{g\lceil_{\gamma}}:\gamma<{\sf dom}(g)\}$ is not in $\Omega_x$. But then for each generic filter $G$ of ${\sf Fn}(\omega_1,\omega,\omega_1)$ which contains $p$, if $h$ is the corresponding generic element, then $B = \{y_{h\lceil{\alpha}}:\alpha<\omega_1\}$ is a member of $\Omega_x$, but no countable subset of it is in $\Omega_x$. But then $\poset = \{g\in{\sf Fn}(\omega_1,\omega,\omega_1):g<p\}$ with the inherited order of ${\sf Fn}(\omega_1,\omega,\omega_1)$ is a countably closed partially ordered set forcing that $X$ is not countably tight.

{\flushleft{ That $(3)\Rightarrow (4)$ and $(4)\Rightarrow(5)$: }} 
These implications follow directly.
{\flushleft{ That $(5)\Rightarrow (6)$: }} A strategy $F$ of ONE together with the fact that $X$ is a space of countable tightness provides an $x$-tightness tree as follows:\\
$F$ calls on ONE to play members of $\Omega_x$. As $X$ is of countable tightness, we may assume that ONE's moves are countable elements of $\Omega_x$. Thus, enumerate $F(\emptyset)$ bijectively as $\{y_{\{(0,n)\}}:n<\omega\}$. For $\alpha<\omega_1$ assume that we have already defined for each $\beta<\alpha$ and each $g\in\,^{\beta}\omega$ a $y_g\in X$ such that $\{y_{g\lceil_{\gamma}\cup\{(\gamma,n)\}}:n<\omega\} = F(y_{g\lceil_0},\cdots,y_{g\lceil_\gamma})\,  \in \Omega_x,\, (\gamma<\beta)$.
{\flushleft{\bf Case 1: }$\alpha = \beta+1$, a successor ordinal. }\\
Then we define $\{y_{g\cup\{(\beta,n)\}}:n<\omega\} = F(y_{g\lceil_0},\cdots,y_{g\lceil_\gamma},\,\cdots,\,y_g)$

{\flushleft{\bf Case 2: }$\alpha$ is a limit ordinal. }\\
In this case we choose $y_g\in F(g\lceil_{\gamma}:\gamma<\alpha)$ arbitrarily.

But then the set $\{y_g:\, g\in\,\cup_{\alpha<\omega_1}\,^{\alpha}\omega\}$ is an $x$-tightness tree. Applying $(5)$ we fix an $f\in\cup_{\alpha<\omega_1}\,^{\alpha}\omega$ such that $\{y_{f\lceil_{\beta}}:\beta<{\sf dom}(f)\}\in \Omega_x$. But then $f$ codes a play of the game against $F$ in which TWO won. This shows that $F$ is not a winning strategy for ONE in $\gone^{\omega_1}(\Omega_x,\Omega_x)$.

{\flushleft{ That $(6)\Rightarrow (1)$: }}\\
Let $(\poset,<)$ be a countably closed partially ordered set and let $X$ be a topological space that is countably tight at $x\in X$. 

Let $\dot{A}$ be a $\poset$ name such that ${\mathbf 1}_{\poset}\forces``\check{x}\in\overline{\dot{A}}"$.
Choose an arbitrary member $p$ of $\poset$. We now use ideas as in the proof of $(2)\Rightarrow(1)$ to define a strategy $F$ of ONE in the game $\gone^{\omega_1}(\Omega_x,\Omega_x)$.

To begin, define $H_{\emptyset} = \{y\in X: (\exists q\le p)(q\forces ``\check{y}\in\dot{A}")\}$.
As in Claim 1 above, $H_{\emptyset}\in\Omega_x$. Since $X$ has countable tightness at $x$, choose $F(\emptyset) = C_{\emptyset}\subseteq H_{\emptyset}$ countable with $x\in\overline{C}_{\emptyset}$. Enumerate $C_{\emptyset}$ as $(y_{(n)}:n<\omega)$. For each $n$ choose $p_{(n)}<p$ such that $p_{(n)}\forces``y_{(n)}\in \dot{A}$". Now $H_{\emptyset}$, $F(\emptyset)$, $p_{(n)},\, n<\omega$ and $y_{(n)}\, n<\omega$ are specified. 

To  describe the rest of the recursive construction of ONE's strategy $F$, suppose that $0<\alpha<\omega_1$ is given, and that for each $\gamma<\alpha$, and each $\sigma\in \,^{\gamma}\omega$ we already have specified:\\
the set $H_{\sigma}$, a countable subset $C_{\sigma}$ of $H_{\sigma}$, element $p_{\sigma}$ of $\poset$, and if $\gamma$ is a successor ordinal, $y_{\sigma}\in X$ such that
\begin{enumerate}
  \item{$H_{\sigma} = \{y\in X:(\exists q\le p_{\sigma})(q\forces ``\check{y}\in\dot{A}")\}$ is in $\Omega_x$;} 
  \item{$F(y_{\nu}:\nu \subset \sigma \mbox{ and }dom(\nu)\mbox{ a successor ordinal}) = C_{\sigma} \subseteq H_{\sigma}$ is a countable set which is a member of $\Omega_x$;} 
  \item{If $\xi = dom(\sigma)<\alpha$ then $C_{\sigma} = \{y_{\sigma\cup\{(\xi,n)\}}:n<\omega\}$.}  
  \item{$p_{\sigma} < p_{\nu}$ for each $\nu\in\,^{\xi}\omega$ with $\nu\subset \sigma$;} 
  \item{If $dom(\sigma)$ is a successor ordinal, then $p_{\sigma}\forces ``y_{\sigma}\in\dot{A}"$;}
\end{enumerate}
We must now specify these parameters for $\tau\in\, ^{\alpha}\omega$.\\

{\flushleft{\underline{Case 1: }} $\alpha = \beta+1$, a successor ordinal.} Consider any $\sigma\in\,^{\beta}\omega$. Since $p_{\sigma}$ is already defined, we have $H_{\sigma} = \{y\in X: (\exists q\le p_{\sigma})(q\forces``\check{y}\in\dot{A}")\}$
where as in Claim 1 above, $H_{\sigma}$ is an element of $\Omega_x$. By the countable tightness of $X$ at $x$ the countable set $C_{\sigma}\subseteq H_{\sigma}$ is selected such that $C_{\sigma}$ is an element of $\Omega_x$, and we define
\[
  F(y_{\nu}:\nu\subset\sigma \mbox{ and }dom(\nu)\mbox{ a successor ordinal}) = C_{\sigma}.
\]
By enumerating $C_{\sigma}$ as $\{y_{\sigma\cup\{(\beta,n)\}}:n<\omega\}$ we specify $y_{\tau}$ for each $\tau$ in $^{\alpha}\omega$ which extends $\sigma$. Then for each of these $y_{\tau}$ we choose a $p_{\tau}<p_{\sigma}$ such that $p_{\tau}\forces``\check{y}_{\tau}\in \dot{A}."$

{\flushleft{\underline{Case 2: }} $\alpha$ is a limit ordinal.} For $\sigma\in\,^{\alpha}\omega$ choose, by the countable closedness of $\poset$, a $p_{\sigma}\in\poset$ such that for each initial segment $\nu$ of $\sigma$ we have $p_{\sigma}<p_{\nu}$ and then define
\[
  H_{\sigma}=\{y\in X:(\exists q\le p_{\sigma})(q\forces``\check{y}\in\dot{A}")\}.
\]
As before $H_{\sigma}$ is a member of $\Omega_x$, and the recursive construction can continue.

Since the defined $F$ is a strategy for ONE in the game $\gone^{\omega_1}(\Omega_x,\Omega_x)$, our hypothesis implies that $F$ is not a winning strategy for ONE. Thus, choose an $F$-play lost by ONE. This play is of the following form: For an $f \in\,^{\omega_1}\omega$ we have the sequence 
\[
   ((F(y_{f\lceil_{\beta}}:\beta<\alpha \mbox{ a successor ordinal}),y_{f\lceil_{\alpha}}):\alpha<\omega_1 \mbox{ a successor 
    ordinal})
\]
for which the set $\{y_{f\lceil_{\alpha}}:\alpha<\omega_1 \mbox{ a successor ordinal}\}$ of player TWO's moves in the play is a member of $\Omega_x$. Using the countable tightness of $X$ at $x$ again, we find that there is a countable ordinal $\beta<\omega_1$ for which the set
\[
  D = \{y_{f\lceil_{\alpha}}:\alpha<\beta \mbox{ a successor ordinal}\}
\]
is in $\Omega_x$. But then we have $p_{f\lceil_{\beta}}\forces``\check{D}\subseteq \dot{A} \mbox{ is a countable element of }\Omega_{\check{x}}."$

Thus, we find that for each $p$ in $\poset$ there is a $q<p$ which forces that $x$ is in the closure of some countable subset of $\dot{A}$. It follows that 
\[
  {\mathbf 1}_{\poset}\forces``\check{X} \mbox{ has countable tightness at }\check{x}." \hspace{0.1in} \Box
\]

 The referee of an earlier version of this paper suggested that there ought to be a common generalization of the above characterization of indestructible countable tightness, and the analogous characterization of Tall's notion of indestructibly Lindel\"of. The following remarks are a small step in this direction.

 Fix two sets $R$ and $S$. Consider the formula $\Phi(S,T)$ which is of the form $(\forall x\in S)(\exists y\in T)\Psi(x,y)$. Define, from $R$ and $S$, the family $\mathcal{A}_{S,R} = \{T\subseteq R:\, \Phi(S,T)\}$. Assume that $\mathcal{A}_{S,R}$ has the properties that 
\begin{itemize}
\item[I]{If $T\in \mathcal{A}_{S,R}$ and $T\subseteq U\subseteq R$, then $U\in\mathcal{A}_{S,R}$, and}
\item[II]{for each element $T$ of $\mathcal{A}_{S,R}$ there is a countable subset $C\subseteq T$ with $C\in\mathcal{A}_{S,R}$.}
\end{itemize}

A function  $F:\,^{<\omega_1}\omega \longrightarrow R$ with the property that for each $\alpha<\omega_1$ and for each $f$ in $^{\alpha}\omega$
\[
  \{F(f\frown(\alpha,n)):n<\omega\} \in \mathcal{A}_{S,R}
\]
is said to be an $\mathcal{A}_{S,R}$-tree.

 Following the ideas in the proof of Theorem \ref{indestrtightchar}, one obtains the following theorem:
\begin{theorem}\label{generalizedTh} Assume that $\Psi(x,y)$ is either of $x\in y$, or $y\in x$. Let $R$ and $S$ be sets. The following statements are equivalent:
\begin{enumerate}
\item{For each countably closed partially ordered set ${\mathbb P}$,  
          \[  
              {\mathbf 1}_{\mathbb P} \forces ``\mbox{Each element of } \dot{\mathcal{A}}_{\check{S},\check{R}}  \mbox{ has a countable subset that is a member of  }\dot{\mathcal{A}}_{\check{S},\check{R}}"
          \]
       }
\item{ ${\sf Fn}(\omega_1,\omega,\omega_1) \forces ``\mbox{Each element of } \dot{\mathcal{A}}_{\check{S},\check{R}} \mbox{ has a countable subset that is a }\\$ $\mbox{member of  }\dot{\mathcal{A}}_{\check{S},\check{R}}"$}
\item{For each $\mathcal{A}_{S,R}$-tree $F$ the set $\{g\in\,^{<\omega_1}\omega: \{F(g\lceil_{\gamma}):\gamma<dom(g)\}\in\mathcal{A}_{S,R}\}$ is dense in ${\sf Fn}(\omega_1,\omega,\omega_1)$.}
\item{For each $\mathcal{A}_{S,R}$-tree $F$ and for each $g\in\,^{<\omega_1}\omega$ there is an $f$ in $^{\omega_1}\omega$ such that $g\subset f$, and $\{F(f\lceil_{\gamma}):\gamma<\omega_1\}$ is a member of $\mathcal{A}_{S,R}$.}
\item{For each $\mathcal{A}_{S,R}$-tree $F$ there is an $f$ in $^{<\omega_1}\omega$ such that $\{F(f\lceil_{\gamma}):\gamma< dom(f)\}$ is a member of $\mathcal{A}_{S,R}$.}
\item{ONE has no winning strategy in the game $\gone^{\omega_1}(\mathcal{A}_{S,R},\mathcal{A}_{S,R})$.}
\end{enumerate}
\end{theorem}

 If $\Psi(x,y)$ is  $x\in y$ we get the indestructibly Lindel\"of notion by taking $R$ to be the underlying space and $S$ to be the topology of the space; if it is $y\in x$, we get the indestructibly countably tight at the point $p$ notion by taking $R$ to be a neighborhood base of the specific point $p$, and $S$ to be the underlying set of the space. It would be interesting to know for which formulae $\Psi(x,y)$ besides these two atomic formulae of the language of set theory one can prove the equivalences of Theorem \ref{generalizedTh}.

Not all is lost when countable tightness is destroyed by a countably closed partial order: A topological space $(X,\tau)$ is said to have \emph{countable extent} if each closed, discrete subspace of $X$ is countable. The Lindel\"of property implies countable extent.  Tall (\cite{Tall}, Lemma 8), and independently Dow \cite{AD1}, proved:
\begin{lemma}[Dow, Tall]\label{tightnessextent} If a Lindel\"of space has countable tightness, then in generic extensions by countably closed partially ordered sets the space has countable extent.
\end{lemma}

\section{Countable strong fan tightness and Cohen reals}

Recall that for families $\mathcal{A}$ and $\mathcal{B}$ of sets the symbol $\sone(\mathcal{A},\mathcal{B})$
denotes the statement that there is for each sequence $(O_n:n\in\naturals)$ of elements of $\mathcal{A}$ a corresponding sequence $(x_n:n\in\naturals)$ such that for each $n$ we have $x_n\in O_n$, and $\{x_n:n\in\naturals\}\in\mathcal{B}$.

In \cite{Sakai} Sakai defined the notion of countable strong fan tightness at $x$, which in our notation is $\sone(\Omega_x,\Omega_x)$. It is clear that if for some countable ordinal $\alpha$ ONE has no winning strategy in the game $\gone^{\alpha}(\Omega_x,\Omega_x)$, then the space has countable strong fan tightness at $x$. It is not in general true that if a space has countable strong fan tightness at a point $x$, then ONE has no winning strategy in the game $\gone^{\omega}(\Omega_x,\Omega_x)$ - see pp. 250 - 251 of \cite{COC3} for an ad hoc example. In Theorem \ref{onevstightness} below we give another example under the Continuum Hypothesis, {\sf CH}.

Theorem \ref{indestrtightchar} implies that spaces where ONE does not have a winning strategy in the game $\gone^{\omega}(\Omega_x,\Omega_x)$ are indestructibly countably tight at $x$. We now show, by rewriting the proof of $(1)\Rightarrow(2)$ of Theorem 13B of \cite{COC3} into the forcing context, that in the generic extension by uncountably many Cohen reals a ground model space that is countably tight at a point $x$ is converted to a space in which ONE has no winning strategy in the game $\gone^{\omega}(\Omega_x,\Omega_x)$. For uncountable cardinals $\kappa$ let $\cohen(\kappa)$ denote the poset for adding $\kappa$ Cohen reals. We use the following lemma of Dow \cite{AD2}, Lemma 5.2:
\begin{lemma}[Dow]\label{tightnesslemma}
Let $\kappa$ be an infinite cardinal and let $X$ be a topological space which is countably tight at $x\in X$. Then
\[
 {\mathbf 1}_{\cohen(\kappa)}\forces ``\check{X} \mbox{ is countably tight at }\check{x}."
\]
\end{lemma}

\begin{theorem}\label{cohenrealsconvertcountablytight} Let $\kappa$ be an uncountable cardinal. If $(X,\tau)$ is a topological space of countable tightness at $x$, then
\[
  {\mathbf 1}_{\cohen(\kappa)}\forces ``\mbox{ONE has no winning strategy in the game }\gone^{\omega}(\Omega_x,\Omega_x)".
\]
\end{theorem}
$\Proof$
Let $\dot{\sigma}$ be a $\cohen(\kappa)$ name such that 
\[
  \mathbf{1}_{\cohen(\kappa)}\forces``\dot{\sigma}\mbox{ is a strategy of ONE in }\gone^{\omega}(\Omega_x,\Omega_x)."
\]
By Lemma \ref{tightnesslemma} $\mathbf{1}_{\cohen(\kappa)}\forces``\check{X}\mbox{ is countably tight at }\check{x}."$ 
Therefore we have
\[
  {\mathbf 1}_{\cohen(\kappa)}\forces ``\dot{\sigma}(\emptyset) \mbox{ has a countable subset which is a member of }\Omega_{\check{x}}."
\]
Choose a $\cohen(\kappa)$ name $\dot{O}_{\emptyset}$ such that 
\[
  {\mathbf 1}_{\cohen(\kappa)}\forces ``\dot{O}_{\emptyset}\subseteq \dot{\sigma}(\emptyset) \mbox{ is a countable subset with $\check{x}$ in its closure}."
\]
Thus choose $\cohen(\kappa)$ names $\dot{y}_n$, $n<\omega$ such that ${\mathbf 1}_{\cohen(\kappa)}\forces ``\dot{C}_{\emptyset}=\{\dot{y}_n:n<\omega\}."$
Then we have 
${\mathbf 1}_{\cohen(\kappa)}\forces ``(\forall n)(\dot{\sigma}(\dot{y}_n) \mbox{ has a countable subset which is in }\Omega_{\check{x}})."$
For each $n$ we choose $\cohen(\kappa)$ names $\dot{C}_n$ and $\dot{y}_{n,k}$, $k<\omega$ such that 
\[
  {\mathbf 1}_{\cohen(\kappa)}\forces ``\dot{C}_{n}\subseteq \dot{\sigma}(\dot{y}_n) \mbox{ is a countable subset with $\check{x}$ in its closure}"
\]
and ${\mathbf 1}_{\cohen(\kappa)}\forces ``\dot{C}_{n}=\{\dot{y}_{n,k}:k<\omega\}"$
and so on.
In this way we find for each finite sequence $n_1,\cdots,n_k$ of elements of $\omega$ $\cohen(\kappa)$ names $\dot{C}_{n_1,\cdots,n_k}$ and $\dot{y}_{n_1,\cdots,n_k}$ such that 
\[
  {\mathbf 1}_{\cohen(\kappa)}\forces ``\{\dot{y}_{n_1,\cdots,n_k,m}:m<\omega\}=\dot{C}_{n_1,\cdots,n_k}"
\]
and
\[
  {\mathbf 1}_{\cohen(\kappa)}\forces ``\dot{C}_{n_1,\cdots,n_k}\subseteq \dot{\sigma}(\dot{y}_{n_1},\cdots,\dot{y}_{n_1,\cdots,n_k})"
\]
and
\[
  {\mathbf 1}_{\cohen(\kappa)}\forces ``\dot{C}_{n_1,\cdots,n_k} \mbox{ is a countable set in }\Omega_{\check{x}}"
\]

Since $\cohen(\kappa)$ has the countable chain condition and each of the names $\dot{y}_{\tau}$ and $\dot{C}_{\tau}$ is a name for a single element of $X$ or a countable set of elements of $X$, there is an $\alpha<\kappa$ such that each of these is a $\cohen(\alpha)$ name. Thus, factoring the forcing as $\cohen(\alpha)*\cohen(\lbrack \alpha,\kappa))$ we may assume that all the named objects are in the ground model. Then, in the generic extension by $\cohen(\lbrack\alpha,\kappa))$ over this ground model there is a function $f\in \,^{\omega}\omega$ such that $f$ is not in any first category set definable from parameters in the ground model.

Now for each neighborhood $V$ of $x$ in the ground model define, in the ground model
$F_V = \{f\in\,^{\omega}\omega:(\forall k)(y_{f\lceil_k}\not\in V)\}$
is first category and is definable from parameters in the ground model only. Thus, in the generic extension by $\cohen(\lbrack \alpha,\kappa))$ $\bigcup\{F_V: \, V \mbox{ a neighborhood of $x$}\} \neq \,^{\omega}\omega$. Choose in this generic extension an $f$ with 
\[
  f\in\,^{\omega}\omega\setminus \bigcup\{F_V: V \mbox{ a ground model neighborhood of }x\}
\]
Then in the generic extension the $\sigma$-play during which TWO selected the sets $y_{f\lceil_{n}}$, $0<n<\omega$ is won by TWO.
This completes the proof that in the generic extension ONE has no winning strategy in the game $\gone^{\omega}(\Omega_x,\Omega_x)$ on $X$.
$\Box$

The reader familiar with the argument in \cite{MSFT} that adding uncountably many Cohen reals over a ground model converts all ground model Lindel\"of spaces to Rothberger spaces would notice that essentially the same argument is used in the proof of Theorem \ref{cohenrealsconvertcountablytight}. Indeed, both results are a special case of the more general fact that for families $\mathcal{A}$ of sets having the property that each element of $\mathcal{A}$ has a countable subset that still is a member of $\mathcal{A}$, and for which this property is preserved by addition of Cohen reals, the property is converted to ONE not having a winning strategy in the game $\gone^{\omega}(\mathcal{A},\mathcal{A})$ in the generic extension by uncountably many Cohen reals.

The property that ONE has no winning strategy in the game $\gone^{\omega}(\Omega_x,\Omega_x)$ is preserved by countably closed forcing:
\begin{theorem}\label{countablyclosed} If $(X,\tau)$ is a topological space for which ONE has no winning strategy in the game $\gone^{\omega}(\Omega_x,\Omega_x)$, then for any countably closed partially ordered set $\poset$, 
\[
  {\mathbf 1}_{\poset}\forces ``\mbox{ONE has no winning strategy in the game }\gone^{\omega}(\Omega_x,\Omega_x)".
\]
\end{theorem}
$\Proof$ Let $(X,\tau)$ be a topological space for which ONE has no winning strategy in the game $\gone^{\omega}(\Omega_x,\Omega_x)$. Let $(\poset,<)$ be a countably closed partially ordered set. Let $\dot{\sigma}$ be a $\poset$-name for a strategy of ONE in the game $\gone^{\omega}(\Omega_x,\Omega_x)$, played in the generic extension, but on the ground model space $X$. Thus,
\[
  {\mathbf 1}_{\poset}\forces``\dot{\sigma} \mbox{ is a strategy of ONE in }\gone^{\omega}(\Omega_{\check{x}},\Omega_{\check{x}}) \mbox{ played on }\check{X}"
\]
We must show
${\mathbf 1}_{\poset}\forces``\dot{\sigma} \mbox{ is not a winning strategy for ONE.}"$
Thus, let $p\in\poset$ be given. We will find a $q<p$ such that $q$ forces that $\dot{\sigma}$ is not a winning strategy for ONE. 

By Theorem \ref{indestrtightchar} we have ${\mathbf 1}_{\poset}\forces``\check{X} \mbox{ is countably tight}."$ Thus, we may assume that 
\[
  {\mathbf 1}_{\poset}\forces``\mbox{ for each finite sequence }(\check{x}_1,\cdots,\check{x}_n) \mbox{ from }\check{X}, \dot{\sigma}(\check{x}_1,\cdots,\check{x}_n) \mbox{ is countable}"
\]
Define $F(\emptyset) = \{y\in X:(\exists q\le p)(q\forces ``\check{y}\in \dot{\sigma}(\emptyset)")\}$. As in Claim 1 of Theorem 1 we have in the ground model that $F(\emptyset)$ is an element of $\Omega_x$, and we may assume that $F(\emptyset)$ is countable. Enumerate $F(\emptyset)$ as $(y_n:n<\omega)$ and choose for each $n$ a $q_n<p$ such that $q_n\forces``\check{y}_n \in \dot{\sigma}(\emptyset)"$. 

Then, for each $n_1$, define $F(y_{n_1}) = \{y\in X:(\exists q\le q_{n_1})(q\forces ``\check{y}\in \dot{\sigma}(\check{y}_{n_1})")\}$. As before the ground model set $F(y_{n_1})$ is an element of $\Omega_x$ and may be assumed countable, and thus may be enumerated as $(y_{n_1,n}:n<\omega)$. Choose for each $n$ a $q_{n_1,n}<q_{n_1}$ such that $q_{n_1,n}\forces``\check{y}_{n_1,n}\in\dot{\sigma}(\check{y}_{n_1})"$.

Next, for each $(n_1,n_2)$ define
\[
  F(y_{n_1},y_{n_1,n_2}) = \{y\in X:(\exists q\le q_{n_1,n_2})(q\forces ``\check{y}\in \dot{\sigma}(\check{y}_{n_1},\check{y}_{n_1,n_2})")\}.
\]
Then the ground model set $F(y_{n_1},y_{n_1,n_2})$ is an element of $\Omega_x$ which may be assumed countable. Enumerate $F(y_{n_1},y_{n_1,n_2})$ as as $(y_{n_1,n_2,n}:n<\omega)$. Choose for each $n$ a $q_{n_1,n_2,n}<q_{n_1,n_2}$ such that $q_{n_1,n_2,n}\forces``\check{y}_{n_1,n_2,n}\in\dot{\sigma}(\check{y}_{n_1}\check{y}_{n_1,n_2})"$, and so on.

In this way we define a strategy $F$ for ONE in the ground model. But in the ground model ONE has no winning strategy in $\gone^{\omega}(\Omega_x,\Omega_x)$. Thus, fix and $F$-play lost by ONE. This specifies an $f\in\,^{\omega}\omega$ such that for each $n$ we have $y_{f\lceil_{n+1}} \in F(y_{f(0)}, \cdots, y_{f\lceil_n})$, and $\{y_{f\lceil_n}:n<\omega\}\in\Omega_x$. But since $\poset$ is countably closed, choose a $q\in\poset$ such that for each $n$ we have $q<q_{f\lceil_n}$, $n<\omega$. Then we have $q<p$ and $q$ forces 
$``(\dot{\sigma}(\emptyset),\check{y}_{f(0)},\dot{\sigma}(\check{y}_{f(0)}), \check{y}_{f(0),f(1)},\, \cdots,\dot{\sigma}(\check{y}_{f(0)},\cdots,\check{y}_{f\lceil_n}), \check{y}_{f\lceil_{n+1}},\cdots)$ is a $\dot{\sigma}$-play lost by ONE." In particular, 
\[
  q\forces``\dot{\sigma} \mbox{ is not a winning strategy for ONE in }\gone^{\omega}(\Omega_{\check{x}},\Omega_{\check{x}})."
\]
This completes the proof. $\Box$

As a result we obtain the following strengthening of \cite{AD2} Lemma 5.6:
\begin{corollary}\label{dow5.6improved}
Let $(X,\tau)$ be a space which is countably tight at $x\in X$. Let $\kappa$ be an uncountable cardinal and let $\dot{\poset}$ be a $\cohen(\kappa)$ name for a countably closed poset. Then
\[
  {\mathbf 1}_{\cohen(\kappa)*\dot{\poset}} \forces``\check{X}\mbox{ has countable strong fan tightness at }\check{x}".
\]
\end{corollary}

\section{Tightness and hereditarily separable spaces}

The following Lemma is well-known: 
\begin{lemma}\label{HSistight}
If $X$ is a hereditarily separable space, then $X$ has countable tightness at each of its elements.
\end{lemma}

W.A.R. Weiss in a personal communication pointed out that with a little more work one can prove that a hereditarily separable space is indestructibly countably tight. I thank Dr. Weiss for the permission to include his argument here.
\begin{theorem}[W.A.R. Weiss]\label{weissth}
If $(X,\tau)$ is a hereditarily separable space, then TWO has a winning strategy in the game $\gone^{\omega_1}(\Omega_x,\Omega_x)$ for each $x\in X$.
\end{theorem}
{\flushleft{\bf Proof: }} TWO's strategy is to choose, in inning $\alpha<\omega_1$, an element $x_{\alpha}$ from the set $O_{\alpha}\in \Omega_x$ chosen by ONE in such a way that $x_{\alpha}$ is not in the closure of the set $\{x_{\beta}:\beta<\alpha\}$. To see that this is a winning strategy, note that a hereditarily dense space cannot contain a bijectively enumerated sequence $(y_{\gamma}:\gamma<\omega_1)$ such that for each $\beta$ we have $y_{\beta}\not\in \overline{\{y_{\alpha}:\alpha<\beta\}}$, since the set $\{y_{\beta}:\beta<\omega_1\}$ has a countable dense subset. Thus, in some inning $\alpha<\omega_1$ TWO is unable to choose an $x_{\alpha}$ of the specified kind. This happens because the set of points already selected by TWO is dense in ONE's set $O_{\alpha}$, and thus is already an element of $\Omega_x$.
$\Box$

We will show in Example 7 (in the last section of the paper) that there are hereditarily separable spaces where there is no countable ordinal $\alpha$ such that TWO has a winning strategy in the game $\gone^{\alpha}(\Omega_x,\Omega_x)$ at an element $x$. In certain special classes of hereditarily separable spaces one can show that there is a countable ordinal $\alpha$ such that TWO has a winning strategy in $\gone^{\alpha}(\Omega_x,\Omega_x)$ at some $x$. We explore this for {\sf HFD} spaces, and for hereditarily separable compact spaces. 

\cite{JuhaszHFD} contains a nice survey of HFD spaces. Note that HFD spaces are subspaces of $\,^{\lambda}2$ for appropriate uncountable cardinals $\lambda$. As noted in 2.7 of \cite{JuhaszHFD}, all HFD's are hereditarily separable.  Thus, by Theorem \ref{weissth}, each HFD is indestructibly countably tight. In the case of HFD's more information can be derived regarding existence of winning strategies of ONE and TWO. 

In \cite{BJ} Berner and Juhasz introduce a game denoted ${\sf G}_{\omega}^{ND}(X)$, which is played as follows on the topological space $X$: Players ONE and TWO play an inning per finite ordinal. In inning $n<\omega$ ONE selects a nonempty open subset $O_n$ of $X$, and TWO responds with a $t_n\in O_n$. A play
\[
  O_0,\, t_0,\, \cdots,\, O_n,\, t_n,\, \cdots
\]
is won by player ONE if the set $\{t_n:n<\omega\}$ is not discrete in $X$. Else, TWO wins. Let $\mathfrak{ND}$ denote the set $\{A\subseteq X: \, A \mbox{ not discrete}\}$, and let $\mathfrak{D}$ denote the set of all dense subsets of $X$. Using the techniques in Lemma 2 and Theorems 7 and 8 of \cite{COC6} one can show the following:
\begin{itemize}
  \item{ONE has a winning strategy in ${\sf G}^{ND}_{\omega}(X)$ if, and only if, TWO has a winning strategy in $\gone^{\omega}(\mathfrak{D},\mathfrak{ND})$.} 
  \item{TWO has a winning strategy in ${\sf G}^{ND}_{\omega}(X)$ if, and only if, ONE has a winning strategy in $\gone^{\omega}(\mathfrak{D},\mathfrak{ND})$.} 
\end{itemize}

\begin{lemma}\label{ndvstight} Let $X$ be a ${\sf T}_1$-space with no isolated points and assume that ONE has a winning strategy in the game $\gone^{\omega}(\mathfrak{D},\mathfrak{ND})$. Then for each $x\in X$ ONE has a winning strategy in the game $\gone^{\omega}(\Omega_x,\Omega_x)$.
\end{lemma}
$\Proof$ Let $\sigma$ be ONE's winning strategy in $\gone^{\omega}(\mathfrak{D},\mathfrak{ND})$. Fix an $x\in X$. Note that we may assume that $\sigma(\emptyset)$ has $x$ as an element. Define a strategy $\sigma_x$ for ONE as follows:\\
$\sigma_x(\emptyset) = \sigma(x)\setminus\{x\}$, a dense subset of $X$ since $X$ has no isolated points. With $(t_1,\cdots,t_n)$ a finite sequence of points from $X$, define $\sigma_x(t_1,\cdots,t_n) = \sigma(x,t_1,\cdots,t_{n-1})\setminus\{x\}$ whenever the latter is defined, and else define this to be $X\setminus\{x\}$.

We claim that $\sigma_x$ is a winning strategy for ONE in $\gone^{\omega}(\Omega_x,\Omega_x)$. For consider any $\sigma_x$-play
\[
  O_0,\, t_0,\, \cdots,\, O_n,\, t_n,\,\cdots
\]
Then 
\[
  \sigma(\emptyset),\,x,\, \sigma(x),\, t_0,\, \cdots,\, \sigma(x,t_0,\cdots,t_{n-1}),\, t_n,\,\cdots
\]
is a $\sigma$-play of $\gone^{\omega}(\mathfrak{D},\mathfrak{ND})$, and thus the set $\{t_n:n<\omega\}\cup\{x\}$ is discrete. Since $x\not\in S = \{t_n:n<\omega\}$, it follows that $S$ is not a member of $\Omega_x$, and so ONE won the play. 
$\Box$

Thus we have
\begin{theorem}[CH]\label{onevstightness} There is a ${\sf T}_{3}$ space $X$ which has countable strong fan tightness and yet ONE has a winning strategy in the game $\gone^{\omega}(\Omega_x,\Omega_x)$ at each $x\in X$.
\end{theorem}
$\Proof$ In Theorem 3.1 of \cite{BJ} Berner and Juhasz construct, using {\sf CH}, an HFD ${\mathbb J}\subseteq \,^{\omega_1}2$ with no isolated points such that ONE has a winning strategy in $\gone^{\omega}(\mathfrak{D},\mathfrak{ND})$. The following Theorem \ref{twowinshfdtightness} implies that each $HFD$ has countable strong fan tightness. Apply this information and Lemma \ref{ndvstight} to the HFD ${\mathbb J}$.
$\Box$

At first glance Theorem \ref{onevstightness} might suggest that there is an HFD with destructible countable tightness. We now modify Theorem 2.7 of \cite{BJ} to show that this is in fact not the case. For the convenience of the reader we recall: 

If $X\subseteq \,^{\lambda}2$ is an HFD define for S a countably infinite subset of X the set $D_S$ to be the set of all countable $A \subseteq \lambda$ such that for all $\sigma \in {\sf Fin}(A, 2)$, if $\lbrack \sigma\rbrack \cap S$ is infinite then for each $\tau \in {\sf Fin}(\lambda \setminus A, 2)$ the set $\lbrack \sigma\rbrack \cap \lbrack\tau \rbrack \cap S$ is nonempty. It is known (see \cite{JuhaszHFD}, 2.12) that:

\begin{lemma} For each countably infinite subset $S$ of an HFD $X\subset \,^{\lambda}2$ the set $D_S\subseteq \lbrack \lambda\rbrack^{\aleph_0}$ is closed and unbounded.
\end{lemma}

\begin{theorem}\label{twowinshfdtightness}
If $X\subset \,^{\lambda}2$ is an HFD then for each $x\in X$ TWO has a winning strategy in $\gone^{\omega^2}(\Omega_x,\Omega_x)$.
\end{theorem}
$\Proof$ Fix an $x\in X$. We may assume that $x$ is not an isolated point of $X$. 

Choose $B_1 \subset \lambda$ countably infinite with $\omega\subseteq B_1$ and $\{\sigma\in {\sf Fin}(B_1, 2): x\in\lbrack \sigma\rbrack\}$  infinite. Let $(\sigma^1_n: 0 < n < \omega)$ enumerate $\{\sigma\in {\sf Fin}(B_1, 2): x\in\lbrack \sigma\rbrack\}$ in such a way that each element is listed infinitely many times. 

We now describe a strategy $\Phi$ for player TWO: During the first $\omega$ innings, when ONE plays in inning $i$ a set $O_i\in\Omega_x$, TWO plays 
\[
  \Phi(O_1,\cdots,O_i)\in \lbrack \sigma^1_i\rbrack \cap O_i\setminus \{\Phi(O_1,\cdots,O_j):j<i\}.
\]

Put $A_1 = \{\Phi(O_1,\cdots,O_k) : 0 < k < \omega\}$. Then $A_1$ is an infinite subset of $X$. If $A_1\in\Omega_x$, then TWO plays the remaining innings arbitrarily, obeying the rules of the game. Thus, assume that $A_1\not\in\Omega_x$. Towards defining $\Phi$ for the next $\omega$ innings of the game, choose $C_1 \in D_{A_1}$ with $B_1$ a proper subset of $C_1$ and let $(\sigma^2_n : 0 < n <\omega)$ enumerate $\{\sigma\in {\sf Fin}(C_1, 2): x\in\lbrack \sigma\rbrack\}$ such that each element is listed infinitely often. Now when ONE plays $O_{\omega+i},\, i<\omega$ TWO plays 
\[
  \Phi(O_{\nu}:\nu\le \omega+i)\in \lbrack \sigma^2_i\rbrack \cap O_{\omega+i}\setminus \{\Phi(O_{\gamma}:\gamma<j):j\le\omega+i\}
\]
which is possible as $A_1\not\in \Omega_x$ and $x$ is not an isolated point of $X$. 

Put $A_2 = A_1 \cup \{\Phi(O_{\gamma}:\gamma<j):j <\omega\cdot 2\}$. If $A_2$ is a member of $\Omega_x$, then TWO plays arbitrary points during the rest of the game, following the rules of the game. Thus, assume that $A_2$ is not a member of $\Omega_x$. Choose $C_2 \in D_{A_1} \cap D_{A_2}$ with $C_1$ a proper subset of $C_2$ and let $(\sigma^3_n : 0 < n < \omega)$ enumerate $\{\sigma\in {\sf Fin}(C_2, 2): x\in\lbrack \sigma\rbrack\}$ such that each element is listed infinitely often.
Now when ONE plays $O_{\omega\cdot 2+i},\, i<\omega$ TWO plays 
\[
  \Phi(O_{\nu}:\nu\le \omega\cdot 2+i)\in \lbrack \sigma^2_i\rbrack \cap O_{\omega+i}\setminus \{\Phi(O_{\gamma}:\gamma<j):j\le\omega\cdot 2+i\}
\]
which is possible as $A_2\not\in \Omega_x$ and $x$ is not an isolated point. 

 Then put $A_3 = A_2 \cup \{\Phi(O_{\nu}:\nu\le \omega\cdot 2+i): i <\omega\}$. Choose $C_3 \in D_{A_1} \cap D_{A_2} \cap D_{A_3}$ with $C_2$ a proper subset of $C_3$, and so on. TWO continues playing like this until an $n$ is reached at which $A_n\in\Omega_x$, and then plays arbitrary points permitted by the rules of the game.

Suppose that for each $n$ we have $A_n\not\in\Omega_x$. Consider $A = \bigcup_{0<k<\omega} A_k$. The set A came about through TWO using the strategy $\Phi$ above over $\omega^2$ innings. 
{\flushleft{\bf Claim: }} $x\in\overline{A}$.\\
To see that $A$ meets each neighborhood of $x$, let a $\sigma \in {\sf Fin}(\lambda,2)$ be given with $x\in\lbrack \sigma\rbrack$. Also, put $C = \cup\{C_n:0<n<\omega\}$. Since each $D_{A_j}$ is closed and unbounded, $C$ is an infinite member of $\cap_{0<n<\omega}D_{A_n}$. Since the containments $C_n\subset C_{n+1}$ are proper and since ${\sf dom}(\sigma)$ is finite, choose the least $k<\omega$ with ${\sf dom}(\sigma)\cap C_k = {\sf dom}(\sigma)\cap C$. Put $\nu = \sigma\lceil_{C_k}$. Now for each $m\ge k$,  $\nu$ was listed infinitely often among $\{\sigma\in {\sf Fin}(C_m, 2): x\in\lbrack \sigma\rbrack\}$, implying that $\lbrack \nu\rbrack\cap A_m$ is infinite and in particular nonempty. This implies that $\lbrack\sigma\rbrack \cap A\neq\emptyset$. It follows that $\Phi$ is a winning strategy for TWO.
$\Box$

\begin{corollary}\label{HFDindestr} Each HFD is indestructibly countably tight.
\end{corollary}
$\Proof$ Apply Theorems \ref{indestrtightchar} and \ref{twowinshfdtightness}. $\Box$

In section 3 \cite{Gamelength} we defined for separable metric spaces $X$ the ordinal function ${\sf tp}_{sf}(X)$ as 
\[
  {\sf tp}_{sf}(X) = \min\{\alpha: \mbox{ TWO has a winning strategy in }\gone^{\alpha}(\Omega_{\mathbf 0},\Omega_{\mathbf 0}) \mbox{ on }{\sf C}_p(X).\}
\]
In the interest of greater generality and at the risk of creating some confusion, we now re-define these ordinal functions as follows: Let $X$ be a topological space and let $x\in X$ be given:
\[
  {\sf tp}_{sf}(X,x) = \min\{\alpha: \mbox{ TWO has a winning strategy in }\gone^{\alpha}(\Omega_x,\Omega_x)\}
\]
As we saw above, when $X$ is an HFD, then ${\sf tp}_{sf}(X,x) = \omega^2$. One might ask if in general each countable ordinal can occur as the ordinal ${\sf tp}_{sf}(X,x)$ at some point $x$ of some space $X$.  First we treat the case of ${\sf tp}_{sf}(X,x) = \omega$:

\begin{theorem}\label{twowinscsf}
Let $(X,\tau)$ be a separable ${\sf T}_3$-space, and let $x$ be an element of $X$. The following are equivalent:
\begin{enumerate}
\item{There is a countable neighborhood basis at $x$ (\emph{i.e.}, $\chi(X,x)=\aleph_0$).}
\item{TWO has a winning strategy in the game $\gone^{\omega}(\Omega_x,\Omega_x)$.}
\end{enumerate}
\end{theorem}
{\flushleft{\bf Proof:}} 
We prove $(2)\Rightarrow (1)$. Let $\sigma$ be a strategy forTWO in the game $\gone^{\omega}(\Omega_x,\Omega_x)$.
{\flushleft{Claim 1: }} For each finite sequence $(A_1,\cdots,A_n)$ of elements of $\Omega_x$ there is an open neighborhood $U$ of $x$ such that for each $z\in U\setminus\{x\}$ there is a $D\in\Omega_x$ for which $z = \sigma(A_1,\cdots,A_n,D)$.

For suppose the contrary, and let $A_1,\cdots,A_n$ be a sequence of elements of $\Omega_x$ witnessing the failure of the Claim. Then for each neighborhood $U$ of $x$ choose a $z_U\in U\setminus\{x\}$ such that for any element $D$ of $\Omega_x$, we have $z\neq \sigma(A_1,\cdots,A_n,D)$. But then $E = \{z_U:U \mbox{ a neighborhood of }x\}$ is an element of $\Omega_x$, and $\sigma(A_1,\cdots,A_n,E)\in E$, contradicting the choice of elements of $E$.

With Claim 1 established we now proceed to define a countable set of neighborhoods of $x$: To begin, choose by Claim 1 a neighborhood $U_{\emptyset}$ of $x$ such that for each $z\in U_{\emptyset}\setminus \{x\}$ there is an element $A_z\in \Omega_x$ such that $z = \sigma(A_z)$. Since $(X,\tau)$ is separable, let $(z_{n}:n<\omega)$ enumerate a countable dense subset of $U_{\emptyset}$, and for each $n$ let $B_{n}\in\Omega_x$ denote $A_{z_{n}}$.

Next, using Claim 1, choose for each $n$ a neighborhood $U_{n}$ of $x$ such that for each $z\in U_{n}\setminus \{x\}$ there is an $A_z\in \Omega_x$ for which $z = \sigma(B_{n},A_z)$. Then as $(X,\tau)$ is separable, choose a countable dense subset $(z_{n,m}:m<\omega)$ of $U_{n}\setminus\{x\}$, and for each $(n,m)$ let $B_{n,m}$ denote $A_{z_{n,m}}$. In general, assume that for each finite sequence $\nu$ of length at most $k$ of elements of $\omega$ we have chosen $z_{\nu\frown n}$, $B_{\nu\frown n}\in \Omega_x$, and neighborhood $U_{\nu}$ of $x$ such that 
\begin{itemize}
\item[A ]{We have $\{z_{\nu\frown n}: n<\omega\}$ a dense subset of $U_{\nu}\setminus\{x\}$;}
\item[B ]{For each $n$, $x_{\nu\frown n} = \sigma(B_{\nu(0)},\, B_{\nu(0),\nu(1)},\, \cdots, B_{\nu},\, B_{\nu\frown n})$.}
\end{itemize}

For each such $\nu$ and $n$ use Claim 1 to select an open neighborhood $U_{\nu\frown n}$ of $x$ such that for each $z\in U_{\nu\frown n}\setminus \{x\}$ there is an $A_z$ such that $z = \sigma(B_{\nu(0)},\, \cdots, B_{\nu\frown n}, A_{z})$.  Then as $(X,\tau)$ is separable, choose a countable dense subset $\{z_{\nu\frown n \frown m}:m<\omega\}$ of $U_{\nu\frown n}\setminus\{x\}$, and for each $m$ let $B_{\nu\frown n\frown m}$ denote $A_{z_{\nu\frown n\frown m}}$. Thus by recursion there is for each finite sequence $\nu$ of elements of $\omega$ points $x_{\nu}$, neighborhoods $U_{\nu}$ of $x$, and elements $B_{\nu}$ of $\omega_x$ satisfying A and B above.

{\flushleft{Claim 2:}} If $\sigma$ is a winning strategy for TWO in $\gone^{\omega}(\Omega_x,\Omega_x)$, then $\{U_{\nu}:\, \nu \in\, ^{<\omega}\omega\}$ is a neighborhood basis of $x$.

For suppose not: Then as $(X,\tau)$ is a ${\sf T}_3$-space there is an open neighborhood $V$ of $x$ such that for each $\nu\in\,^{<\omega}\omega$, we have $U_{\nu}\setminus \overline{V}$ is nonempty. Then choose $n_1,\, n_2,\, \cdots, n_k,\,\cdots$ so that $z_{n_1}\in U_{\emptyset}\setminus\overline{V}$, $z_{n_1,n_2}\in U_{n_1}\setminus \overline{V}$, and in general $z_{n_1,\cdots,n_{k+1}}\in U_{n_1,\cdots,n_k}\setminus \overline{V}$. But then
\[
  B_{n_1},\, \sigma(B_{n_1}),\, \cdots,\, B_{n_1,\,\cdots,\, n_k},\, \sigma(B_{n_1},\,\cdots,\, B_{n_1,\cdots,n_k}),\, \cdots
\]
is a play won by ONE since none of TWO's responses is in the neighborhood $V$ of $x$. This contradicts the hypothesis that $\sigma$ is a winning strategy for TWO.
$\Box$

Since first countability is preserved in any generic extension, it follows that for ${\sf T}_3$ separable spaces the fact that TWO has a winning strategy in $\gone^{\omega}(\Omega_x,\Omega_x)$ is likewise preserved.

It is not clear what the values of  ${\sf tp}_{sf}(X,x)$ could be in the case of compact hereditarily separable spaces. ${\sf MA} + \neg{\sf CH}$ implies that each compact hereditarily separable space is first countable, whence we have:
\begin{proposition}\label{cptHS}
{\sf MA} $+ \neg{\sf CH}$ implies that in every compact hereditarily separable space, TWO has a winning strategy in $\gone^{\omega}(\Omega_x,\Omega_x)$ at each $x\in X$.
\end{proposition}
On the other hand, V.V. Fedorchuk proved that $\diamondsuit$ implies the existence of a compact hereditarily separable ${\sf T}_2$-space $X$ which has no convergent sequences. By Theorem \ref{twowinscsf} ${\sf tp}_{sf}(X,x)>\omega$ for each $x\in X$. We discuss this space further in Example 6 below. 
 
Our next result shows that there are strong limitations on which infinite ordinals can occur as a ${\sf tp}_{sf}(X,x)$. Recall that an ordinal $\alpha$ is \emph{additively indecomposable} if $\beta+\gamma<\alpha$ whenever $\beta,\, \gamma < \alpha$.
\begin{proposition}\label{sfanaddindec}
Let $X$ be a topological space and let $x$ be an element of $X$. If the ordinal ${\sf tp}_{sf}(X,x)$ is infinite, then it is additively indecomposable.
\end{proposition}
\begin{proof}
Let $X$ be a space and let $x\in X$ be a point for which $\alpha = {\sf tp}_{sf}(X,x)$ is infinite. We may assume that $\alpha>\omega$. If $\beta$ and $\gamma$ are ordinals less than $\alpha$, then $\beta+\gamma<\alpha$. To see this, consider a strategy $F$ of TWO in the game $\gone^{\beta+\gamma}(\Omega_x,\Omega_x)$ on $X$. Since $F$ is also a strategy for TWO in the game $\gone^{\beta}(\Omega_x,\Omega_x)$, and since $\beta<\alpha$ there is a play, fixed from now on, of the form
\[
  A_0,\, F(A_0),\, \cdots,\, A_{\delta},\, F(A_{\nu}:\nu\le\delta),\,\cdots\,\,\delta<\beta
\] 
of $\gone^{\beta}(\Omega_x,\Omega_x)$ on $X$ lost by TWO. Since this play is lost by TWO, choose a neighborhood $U_0$ of $x$ such that $U_0$ is disjoint from the set of all points chosen so far by TWO, using strategy $F$.

Next we define a strategy $G$ for TWO in the game $\gone^{\gamma}(\Omega_x,\Omega_x)$ by setting for each $\mu<\gamma$:
\[
  G(B_{\nu}:\nu<\mu):=F((A_{\rho}:\rho<\beta)\frown(B_{\nu}:\nu<\mu)).
\]
Since $\gamma$ is less than $\alpha$, $G$ is not a winning strategy for TWO, and thus we find a $G$-play
\[
  B_0,\, G(B_0),\, \cdots, B_{\mu},\, G(B_{\nu}:\nu\le\mu),\, \cdots\,\,\mu<\gamma
\]
lost by player TWO. Since this play is lost by TWO choose a neighborhood $U_1$ of $x$ such that none of the points in moves by TWO in this play meets the set $V_1$. But then we have found an $F$-play of $\gone^{\beta+\gamma}(\Omega_x,\Omega_x)$ lost by TWO since no element of the set selected by TWO meets the neighborhood $U_0\cap U_1$ of $x$.
\end{proof}

\section{Homogeneous compacta.} 

In Theorem 2.8 of \cite{JNS} it is proved that in the generic extension, obtained by forcing with $\aleph_2$ Cohen reals, every homogeneous compact ${\sf T}_5$ space is countably tight and of character at most $\aleph_1$. We now show:
\begin{theorem}\label{t5compacta}
In the generic extension obtained by forcing with $\cohen(\aleph_2)$, for each homogeneous compact ${\sf T}_5$-space $X$, ONE has no winning strategy in $\gone^{\omega}(\Omega_x,\Omega_x)$ at each $x\in X$.
\end{theorem}
$\Proof$ 
Let $X$ be a homogeneous compact ${\sf T}_5$-space in the generic extension, let $x\in X$ be given, and let $F$ be a strategy of ONE in the game $\gone^{\omega}(\Omega_x,\Omega_x)$ on $X$. All these items are members of the generic extension. Since $X$ is countably tight in this generic extension, we may assume that $F$ calls on ONE to play countable elements of $\Omega_x$.

Thus we define the following: $O_{\emptyset} = F(\emptyset)$ is ONE's first move. Enumerate it bijectively as $(t_n:n<\omega)$. Then for each $n$ define $O_n = F(t_n)$, and enumerate it bijectively as $(t_{n,m}:m<\omega)$. For $n_1$ and $n_2$ define $O_{n_1,n_2} = F(t_{n_1},\,t_{n_1,n_2})$, and so on. In this we we obtain in the generic extension families $(O_{\sigma}: \, \sigma\,\in\, ^{<\omega}\omega)$ of countable elements of $\Omega_x$. Also for each $\sigma\in\,^{<\omega}\omega$ the set $\{t_{\sigma\frown n}:n<\omega\}$ is an enumeration of the elements of $O_{\sigma}$. The $\cohen(\aleph_2)$ names for $x$, $F$, and the individual $O_{\sigma}$ and $t_{\sigma}$, as well as finite sequences of these, are indeed $\cohen(\alpha)$-names for an $\alpha<\omega_2 (=\aleph_2)$, and thus are present in an initial segment of the generic extension. 

In this initial segment of the generic extension define for each neighborhood $U$ of $x$ the set 
\[
  S_U = \{f\in\,^{\omega}\omega: (\forall n)(t_{f\lceil_n} \not \in U)\}.
\]
Each of the sets $S_U$ is a nowhere dense subset of $^{\omega}\omega$ defined in terms of parameters of the initial segment of the generic extension. Thus, subsequently added Cohen reals are not elements of any $S_U$. Let $c$ be such a Cohen real. Then the play
\[
  F(\emptyset),\, t_{(c(0)},\, F(t_{(c(0))}),\, t_{(c(0),\, c(1))},\, F(t_{(c(0))},\, t_{(c(0),\, c(1))},\, \cdots
\]
is lost by ONE. 
$\Box$

Thus, in the Cohen model all homogeneous ${\sf T}_5$ compacta have countable strong fan tightness and are indestructibly countably tight. 
More can be deduced from a various axioms: A. Dow proved in \cite{AD2}, Theorem 6.3, that {\sf PFA} implies that a compact space of countable tightness has an element with a countable neighborhood base. Thus
\begin{proposition} {\sf PFA} implies that for each homogeneous compact ${\sf T}_2$-space, TWO has a winning strategy in $\gone^{\omega}(\Omega_x,\Omega_x)$ at each $x\in X$.
\end{proposition}
Corollary 4.17 of \cite{dlVega} states that $2^{\aleph_0}<2^{\aleph_1}$ implies that a compact homogeneous space has countable tightness if, and only if, it is first countable. Thus,
\begin{proposition}
$2^{\aleph_0}<2^{\aleph_1}$ implies: A homogeneous compact space has countable tightness if, and only if, TWO has a winning strategy in $\gone^{\omega}(\Omega_x,\Omega_x)$ at each $x\in X$.
\end{proposition}
And by Theorem 4.18 of \cite{dlVega}, 
\begin{proposition}
It is consistent, relative to the consistency of ${\sf ZFC}$, that at each point $x$ of a homogeneous compact ${\sf T}_5$-space TWO has a winning strategy in the game $\gone^{\omega}(\Omega_x,\Omega_x)$. 
\end{proposition}

\section{Generic left-separated spaces.}

Let $\nu$ be an ordinal. In Section 2 of \cite{JS} it is proved that in the generic extension ${\sf V}^{\poset_{\nu}}$ by a special complete suborder $\poset_{\nu}$ of the Cohen partially ordered set ${\sf Fn}(\nu\times\nu,2)$, there is a topology $\tau$ on $\nu$ such that the space $X_{\nu} = (\nu,\tau)$ has the following properties:
\begin{itemize}
  \item[(a)]{$X_{\nu}$ is hereditarily Lindel\"of (Lemma 2.1);} 
  \item[(b)]{$X_{\nu}$ has countable tightness (Lemma 2.2);} 
  \item[(c)]{The density of $X_{\nu}$ is ${\sf cof}(\nu)$ (Lemma 2.3);} 
  \item[(d)]{$X_{\nu}$ is left-separated in the natural well-ordering of $\nu$.} 
\end{itemize}
We will now argue that if $\nu$ is an uncountable regular cardinal, then $X_{\nu}$ is (1) an {\sf L}-space which is hereditarily a Rothberger space (and thus indestructibly Lindel\"of and  a {\sf D}-space), and (2) ONE has no winning strategy in the game $\gone^{\omega}(\Omega_x,\Omega_x)$ at each $x\in X_{\nu}$, so that $X_{\nu}$ in indestructibly countably tight and has countable strong fan tightness.

The fact that $X_{\nu}$ is an {\sf L}-space when $\nu$ is a regular uncountable cardinal was noted in \cite{JS}. Regarding the rest of the claimed properties of $X_{\nu}$:

The first point is that if $\kappa$ (previously $\nu$) is a regular uncountable cardinal then the partial order $\poset_{\kappa}$ densely embed into ${\sf Fn}(\kappa,2)$, so that the generic extension is the same as the generic extension by $\cohen(\kappa)$. Thus $X_{\kappa}$ is obtained by adding $\kappa$ Cohen reals. To see this we review the definition of $\poset_{\kappa}$. An element $p$ of ${\sf Fn}(\kappa\times\kappa,2)$ is an element of $\poset_{\kappa}$ if, and only if, it has the following two properties:
\begin{enumerate}
  \item{If $(\alpha,\alpha)$ is in the domain of $p$, then $p(\alpha,\alpha)=1$;} 
  \item{If $(\alpha,\beta)$ is in the domain of $p$ and $p(\alpha,\beta)=1$, then $\alpha\le \beta$.}  
\end{enumerate}
Now let $S$ be the subset of $\kappa\times\kappa$ consisting of the pairs $(\alpha,\beta)$ with $\alpha<\beta$. Then the mapping 
\[
  F:P_{\kappa}\longrightarrow {\sf Fn}(S,2)
\]
defined by $F(p) = p\lceil_{S}$, the restriction of $p$ to $S$, is a dense embedding (See \cite{Kunen}, Definition VII.7.7). Thus, as $\vert S\vert = \kappa$, the generic extension by $\poset_{\kappa}$ is the same as the generic extension by $\cohen(\kappa)$.

The second point is that since $\kappa$ is regular and uncountable, and $X_{\kappa}$ has countable tightness, An argument as in the proof of Theorem \ref{t5compacta}  
shows that in the generic extension ONE has no winning strategy in $\gone^{\omega}(\Omega_x,\Omega_x)$ at each $x$ in $X_{\kappa}$. Then Theorem \ref{indestrtightchar} implies that $X_{\kappa}$ is indestructibly countably tight. Also apply Theorem 11 of \cite{MSFT} to conclude that $X_{\kappa}$ is Rothberger, and Corollary 10 of \cite{MSFT} to conclude that $X_{\kappa}$ is indestructibly Lindel\"of. By results of Aurichi \cite{Au} the fact that $X_{\kappa}$ is Rothberger implies that $X_{\kappa}$ is a {\sf D}-space. 

\section{Countable tightness in ${\sf C}_p(X)$}

The selection principle $\sfin(\mathcal{A},\mathcal{B})$ is defined as: For each sequence $(A_n:n\in\naturals)$ of elements of $\mathcal{A}$ there is a corresponding sequence $(B_n:n\in\naturals)$ such that for each $n$, $B_n$ is a finite subset of $A_n$, and $\bigcup\{B_n:n\in\naturals\}$ is an element of $\mathcal{B}$.  Arkhangel'skii introduced the notion of countable fan tightness. In the notation of selection principles, a space is said to be \emph{countably fan tight at} $x$ if the selection principle $\sfin(\Omega_x,\Omega_x)$ holds. 
The implications 
\[
  \sone(\mathcal{A},\mathcal{B})\Rightarrow \sfin(\mathcal{A},\mathcal{B})\Rightarrow \mbox{ each member of }\mathcal{A}\mbox{ has a countable subset that is in }\mathcal{B}
\]
always hold. It follows that countable strong fan tightness implies countable fan tightness which in turn implies countable tightness. These three properties have been extensively investigated in the function spaces ${\sf C}_p(X)$, where a nice duality theory relating tightness properties of ${\sf C}_p(X)$ to covering properties of $X$ emerged. There are natural games corresponding to the selection principles $\sone(\mathcal{A},\mathcal{B})$ and $\sfin(\mathcal{A},\mathcal{B})$. We have already encountered the games $\gone^{\alpha}(\mathcal{A},\mathcal{B})$ of ordinal length $\alpha$ earlier in this paper. The game $\gfin^{\alpha}(\mathcal{A},\mathcal{B})$ proceeds as follows: Players ONE and TWO play an inning per ordinal $\beta<\alpha$. In an inning $\beta$ ONE first selects an element $O_{\beta}$ from $\mathcal{A}$. Then TWO responds by selecting a finite set $T_{\beta}\subseteq O_{\beta}$. A play
\[
  O_0,\, T_0,\, \cdots,\,O_{\beta},\, T_{\beta},\,\cdots\, \beta<\alpha
\]
is won by TWO if $\bigcup\{T_{\beta}:\beta<\alpha\}$ is an element of $\mathcal{B}$; otherwise, ONE wins.

The games $\gone^{\omega}(\Omega_x\Omega_x)$ related to countable strong fan tightness, and $\gfin^{\omega}(\Omega_x,\Omega_x)$ related to countable fan tightness were introduced in \cite{COC3}, where they were investigated in the context of ${\sf C}_p(X)$. To describe this background material we introduce the following:

Letting $\open$ denote the family of open covers of a topological space $X$, the space is said to have the \emph{Menger}  property if $\sfin(\open,\open)$ holds; it is said to have the \emph{Rothberger} property if $\sone(\open,\open)$ holds. W. Hurewicz proved in 1925 that a topological space $X$ has the Menger property if, and only if, ONE does not have a winning strategy in the game $\gfin^{\omega}(\open,\open)$, and Pawlikowski proved in 1993  that a topological space has the Rothberger property if, and only if, ONE has no winning strategy in the game $\gone^{\omega}(\open,\open)$. According to Gerlits and Nagy an open cover $\mathcal{U}$ of a space $X$ is an $\omega$-cover if $X$ is not a member of $\mathcal{U}$, but for each finite subset $F$ of $X$ there is a $U\in\mathcal{U}$ such that $F\subseteq U$. We let $\Omega$ denote the collection of open $\omega$-covers of $X$.

The following results are known in connection with player ONE:
\begin{theorem} Let $X$ be a ${\sf T}_{3.5}$-space. 
\begin{itemize}
 \item[A]{(\cite{COC3}, Theorem 11) The following statements are equivalent:
      \begin{enumerate}
        \item{ ${\sf C}_p(X)$ has countable fan tightness.}  
        \item{ONE has no winning strategy in the game $\gfin^{\omega}(\Omega_{\mathbf 0},\Omega_{\mathbf 0})$ on ${\sf C}_p(X)$.}  
        \item{ONE has no winning strategy in the game $\gfin^{\omega}(\Omega,\Omega)$ on $X$.}  
        \item{$X$ has the property $\sfin(\Omega,\Omega)$.} 
      \end{enumerate}
  }
\item[B]{(\cite{COC3}, Theorem 13) The following statements are equivalent:
      \begin{enumerate}
        \item{${\sf C}_p(X)$ has countable strong fan tightness.}  
        \item{ONE has no winning strategy in the game $\gone^{\omega}(\Omega_{\mathbf 0},\Omega_{\mathbf 0})$ on ${\sf C}_p(X)$.}   
        \item{ONE has no winning strategy in the game $\gone^{\omega}(\Omega,\Omega)$ on $X$.}  
        \item{$X$ has the property $\sone(\Omega,\Omega)$.} 
      \end{enumerate}}
\end{itemize}
\end{theorem}

Interest in the existence of winning strategies of player TWO has been growing -see for example \cite{ABPR} - since the works \cite{BD1} and \cite{BD2} of Barman and Dow on the topic. In \cite{BD2}  Theorem 3.6 Barman and Dow show that if the ${\sf T}_{3.5}$ space $X$ is $\sigma$-compact, then TWO has a winning strategy in $\gfin^{\omega}(\Omega_{\mathbf 0},\Omega_{\mathbf 0})$ on ${\sf C}_p(X)$. In Theorem 2 of \cite{ABPR} Bella proves for Tychonoff spaces $X$ that if TWO has a winning strategy in the game $\gfin^{\omega}(\Omega_{\mathbf 0},\Omega_{\mathbf 0})$ on ${\sf C}_p(X)$, then TWO has a winning strategy in the game $\gfin^{\omega}(\Omega,\Omega)$ on $X$.  In Corollary 2 of \cite{ABPR} Bella points out that if the ${\sf T}_{3.5}$ space $X$ has the property that each closed set is a ${\sf G}_{\delta}$-set, then the converse also holds. We shall now show that the hypothesis that each closed subset of $X$ is a ${\sf G}_{\delta}$-set is not needed, and we shall also extend these results to larger countable ordinals.

Corresponding work for games of the form $\gone^{\alpha}(\mathcal{A},\mathcal{B})$ was done in \cite{Gamelength}. Along the lines of that work we define for a topological space $X$:
\[
  {\sf tp}_{\omega-{\sf fin}}(X) := \min\{\alpha: \mbox{ TWO has a winning strategy in the game $\gfin^{\alpha}(\Omega,\Omega)$ on $X$}\}
\]
First we derive the analogue of Theorem 4 (1) of \cite{Gamelength} . Recall that an ordinal number $\alpha>0$ is \emph{additively indecomposable} if for all ordinals $\beta,\, \gamma<\alpha$ we have $\beta+\gamma<\alpha$.

\begin{proposition}\label{finaddindec}
Let $X$ be a topological space. If the ordinal ${\sf tp}_{\omega-{\sf fin}}(X)$ is infinite, then it is additively indecomposable.
\end{proposition}
\begin{proof}
Let $X$ be a space for which $\alpha = {\sf tp}_{\omega-{\sf fin}}(X)$ is infinite. We may assume that $\alpha>\omega$. Let $\beta$ and $\gamma$ be ordinals less than $\alpha$. We show that $\beta+\gamma<\alpha$. To this end, consider a strategy $F$ of TWO in the game $\gfin^{\beta+\gamma}(\Omega,\Omega)$ on $X$. Note that $F$ is also a strategy for TWO in the game $\gfin^{\beta}(\Omega,\Omega)$, and that since $\beta<\alpha$ there is a play, fixed from now on, of the form
\[
  \mathcal{U}_0,\, F(\mathcal{U}_0),\, \cdots,\, \mathcal{U}_{\delta},\, F(\mathcal{U}_{\nu}:\nu\le\delta),\,\cdots\,\,\delta<\beta
\] 
of $\gfin^{\beta}(\Omega,\Omega)$ on $X$ lost by TWO. Since this play is lost by TWO, choose a finite set $F_0\subset X$ such that $F_0$ is not a subset of any of the sets in TWO's responses in this game.

Next we define a strategy $G$ for TWO in the game $\gfin^{\gamma}(\Omega,\Omega)$ by setting for each $\mu<\gamma$:
\[
  G(\mathcal{V}_{\nu}:\nu<\mu):=F((\mathcal{U}_{\rho}:\rho<\beta)\frown(\mathcal{V}_{\nu}:\nu<\mu)).
\]
Since $\gamma$ is less than $\alpha$, $G$ is not a winning strategy for TWO, and thus we find a $G$-play
\[
  \mathcal{V}_0,\, G(\mathcal{V}_0),\, \cdots, \mathcal{V}_{\mu},\, G(\mathcal{V}_{\nu}:\nu\le\mu),\, \cdots\,\,\mu<\gamma
\]
lost by player TWO. Since this play is lost by TWO choose a finite set $F_1\subset X$ such that none of the sets in moves by TWO in this play contains the set $F_1$. But then we have found an $F$-play of $\gfin^{\beta+\gamma}(\Omega,\Omega)$ lost by TWO since no element of the set selected by TWO contains the finite subset $F_0\cup F_1$ of $X$.
\end{proof}

Define for a space $X$ and a point $x\in X$ the ordinal functions
\[
  {\sf tp}_{fin-ft}(X,x) = \min\{\alpha: \mbox{ TWO has a winning strategy in the game }\gfin^{\alpha}(\Omega_{x},\Omega_{x})\}
\]
and 
\[
  {\sf tp}_{fin-ft}(X) = \sup\{ {\sf tp}_{fin-ft}(X,x) : x\in X)\}.
\]
Using a similar argument as in the proof of Proposition \ref{sfanaddindec} we find:
\begin{proposition}\label{fanaddindec}
Let $X$ be a topological space and let $x$ be an element of $X$. If the ordinal ${\sf tp}_{fin-ft}(X,x)$ is infinite, then it is additively indecomposable.
\end{proposition}

Next we derive an analogue of Theorem 4 (2) of \cite{Gamelength}
\begin{proposition}\label{typofstrfortwo}
Let $X$ be a topological space such that ${\sf tp}_{fin-\omega}(X) = \alpha\ge \omega$. Then TWO has a winning strategy $G$ in $\gfin^{\alpha}(\Omega,\Omega)$ which has the following property:
\begin{quote}
{\tt For every sequence $(\mathcal{U}_{\gamma}:\gamma<\alpha)$ of $\omega$-covers of $X$, and for each $\gamma<\alpha$ the set 
$\bigcup \{G(\mathcal{U}_{\nu}:\nu\le\rho): \gamma <\rho<\alpha\}$
is an $\omega$-cover of $X$.}
\end{quote}
\end{proposition}
\begin{proof}
Use the same argument as in the proof of Theorem 4 (2) of \cite{Gamelength}.
\end{proof}

\begin{theorem}\label{twowinsduality}
Let $X$ be a ${\sf T}_{3.5}$ space and let $\alpha$ be an ordinal.  Then the following are equivalent: 
   \begin{enumerate}
   \item{TWO has a winning strategy in the game $\gfin^{\alpha}(\Omega_{\mathbf 0},\Omega_{\mathbf 0})$ on ${\sf C}_p(X)$.}
   \item{TWO has a winning strategy in the game $\gfin^{\alpha}(\Omega,\Omega)$ on $X$.}
   \end{enumerate}
\end{theorem}
\begin{proof}
{\flushleft{\underline{(1) implies (2):}}} 
Let $\alpha$ be a countable ordinal and assume that $\Phi$ is a winning strategy of TWO in the game $\gfin^{\alpha}(\Omega_{\mathbf 0},\Omega_{\mathbf 0 })$. Let $\prec_p$ be a well-ordering of ${\sf C}_p(X)$, and let $\prec_2$ be a well-ordering of $\tau$, the topology of $X$. We define a strategy $F$ for TWO in the game $\gfin^{\alpha}(\Omega,\Omega)$ on $X$. The diagram below may be useful in following how the strategy $F$ is defined from the strategy $\Phi$:

\begin{center}
\begin{tabular}{c|l}
\multicolumn{2}{c}{$\gfin^{\alpha}(\Omega,\Omega)$}\\ \hline\hline
\multicolumn{2}{c}{}\\
ONE              & TWO   ($F$)                                  \\ \hline
$\mathcal{U}_0$   &                                       \\
                 &                                         \\
                 & $\mathcal{V}_0 = F(\mathcal{U}_0)$ \\
$\mathcal{U}_1$   &                                       \\
                 &                                         \\
                 & $\mathcal{V}_1 = F(\mathcal{U}_0,\mathcal{U}_1)$ \\
$\vdots$         & $\vdots$ \\
\end{tabular}
\hspace{1in}
\begin{tabular}{c|l}
\multicolumn{2}{c}{$\gfin^{\alpha}(\Omega_{\mathbf 0},\Omega_{\mathbf 0})$}                        \\ \hline\hline
\multicolumn{2}{c}{}\\ 
ONE   & TWO ($\Phi$)                                                \\ \hline
      &                                                     \\
$A_0$ & $B_0 = \Phi(A_0)$                                      \\
      &                                                     \\
      &                                                     \\
$A_1$ & $B_1=\Phi(A_0,A_1)$                                    \\
      &                                                     \\
$\vdots$         & $\vdots$ \\
\end{tabular}
\end{center}

Let $\mathcal{U}_0$ be ONE's first move in $\gfin^{\alpha}(\Omega,\Omega)$. For a finite set $F\subset X$ and an open set $U\subset X$ with $F\subset U$, let $\chi(F,U)$ be the $\prec_2$-first element $f$ of ${\sf C}_p(X)$ such that $f\lceil_{F}\equiv 0$ and $f\lceil_{X\setminus U}\equiv 1$ (which exists as $X$ is completely regular). Define ONE's move in $\gfin^{\alpha}(\Omega_{\mathbf 0},\Omega_{\mathbf 0})$ to be $A_0 = \{\chi(F,U): U\in\mathcal{U}_0 \mbox{ and } F\subset U \mbox{ finite}\}$. Note that indeed $A_0$ is an element of $\Omega_{\mathbf 0}$. From TWO's response $\Phi(A_0)$ we define $F(\mathcal{U}_0)$ as follows: Suppose that $\Phi(A_0)$ is the finite set $\{\chi(F_1,U_1),\cdots,\chi(F_n,U_n)\}$. Define $\mathcal{V}_0 = F(\mathcal{U}_0)$ to be $\{U_1,\cdots,U_n\}$.

Suppose that in inning $\beta<\alpha$ ONE moves $\mathcal{U}_{\beta}$ in the game $\gfin^{\alpha}(\Omega,\Omega)$. As above, compute $A_{\beta} = \{\chi(F,U): U\in\mathcal{U}_{\beta} \mbox{ and }F\subset U \mbox{ finite}\}$.  Apply TWO's strategy $\Phi$ and compute $B_{\beta} = \Phi(A_{\gamma}:\gamma\le \beta)$, a finite subset of $A_{\beta}$, and define the corresponding finite set $\mathcal{V}_{\beta} = F(\mathcal{U}_{\gamma}:\gamma\le \beta)$ as follows: With $B_{\beta} = \{\chi(F_1,U_1),\cdots,\chi(F_n,U_n)\}$, set $\mathcal{V}_{\beta}$ equal to $\{U_1,\cdots, U_n\}$. 

This defines TWO's strategy $F$ in the game $\gfin^{\alpha}(\Omega,\Omega)$. To see that $F$ is a winning strategy for TWO, consider an $F$-play
\[
  \mathcal{U}_0,\, F(\mathcal{U}_0),\cdots,\,\mathcal{U}_{\beta},\, F(\mathcal{U}_{\gamma}:\gamma\le \beta),\cdots, \,\, \beta<\alpha.
\]
Also, consider a fixed finite subset $F$ of $X$. We must show that some element of $\bigcup_{\beta<\alpha}\mathcal{V}_{\beta}$ contains the set $F$.

Consulting the definition of $F$ above we see that this play corresponds to a $\Phi$-play 
\[
 A_0,\, \Phi(A_0),\, \cdots,\, A_{\beta},\, \Phi(A_{\gamma}:\gamma\le \beta),\cdots,\,\,\beta<\alpha
\]
of $\gfin^{\alpha}(\Omega_{\mathbf 0},\Omega_{\mathbf 0})$ where for each $\beta<\alpha$:
\begin{enumerate}
\item{$A_{\beta} = \{\chi(F,U):U\in\mathcal{U}_{\beta} \mbox{ and }F\subset U\mbox{ finite}\}$ and}
\item{With $F(A_{\gamma}:\gamma\le \beta) = \{\chi(F_1,U_1),\cdots,\chi(F_n,U_n)\}$, $F(\mathcal{U}_{\gamma}:\gamma\le\beta) = \{U_1,\cdots,U_n\}$.}
\end{enumerate}
Since $\Phi$ is a winning strategy for TWO, for each $\epsilon>0$ the neighborhood $\lbrack F,\epsilon\rbrack = \{f\in{\sf C}_p(X): (\forall x\in F)(\vert f(x)\vert<\epsilon)\}$ has a nonempty intersection with $\{\Phi(A_{\gamma}:\gamma\le\beta\}):\beta<\alpha\}$. Fix $\epsilon<1$ and choose a $\beta<\alpha$ such that 
\[
  \lbrack F,\epsilon\rbrack \bigcap \Phi(A_{\gamma}:\gamma\le\beta\}) \neq\emptyset,
\]
and choose an element $\chi(F_1,U_1)$ of this set. Note that $U_1$ is an element of $\mathcal{V}_{\beta}$. We claim that $F\subseteq U_1$. For suppose not, and choose $x\in F\setminus U_1$. Then we have $\chi(F_1,U_1)(x) = 1$, while also $\vert \chi(F_1,U_1)(x)\vert <\epsilon<1$, a contradiction. 

This completes the proof that $F$ is a winning strategy for TWO.

{\flushleft{\underline{(2) implies (1):}}} 
Let $F$ be a given winning strategy of TWO in the game $\gfin^{\alpha}(\Omega,\Omega)$. By Proposition \ref{finaddindec}, $\alpha$ is additively indecomposable, and thus is a limit ordinal. As $\alpha$ is countable we can choose ordinals $\gamma_0<\gamma_1<\cdots <\gamma_n<\cdots <\alpha$, $n<\omega$, that converge to $\alpha$. For each $n<\omega$, put $\epsilon_n = \frac{1}{2^{n+1}}$. We may further assume that $F$ has the property stated in Proposition \ref{typofstrfortwo}.

Since $\alpha$ is a countable ordinal $X$ has the property $\sfin(\Omega,\Omega)$, whence by a theorem of Arkhangel'skii \cite{Arkh}, ${\sf C}_p(X)$ has countable tightness. Thus we may assume that each move made by ONE in the game $\gfin^{\alpha}(\Omega,\Omega)$ is a countable subset of ${\sf C}_p(X)$. We may also assume by Lemma 10 of \cite{COC3} that if $A$ is a move by ONE, then each $f\in A$ is pointwise nonnegative, and for $0<\epsilon<1$,  $\{\{x\in X: f(x)<\epsilon\}: f\in A\}$ is an $\omega$-cover of $X$.

We define a strategy $\Phi$ for TWO in the game $\gfin^{\alpha}(\Omega_{\mathbf 0 },\Omega_{\mathbf 0 })$ on $X$. The diagram below may be useful in following how the strategy $\Phi$ is defined from the strategy $F$:

\begin{center}
\begin{tabular}{c|l}
\multicolumn{2}{c}{$\gfin^{\omega}(\Omega_{\mathbf 0 },\Omega_{\mathbf 0 })$}\\ \hline\hline
\multicolumn{2}{c}{}\\
ONE              & TWO   ($\Phi$)                                  \\ \hline
$A_0$   &                                       \\
                 &                                         \\
                 & $B_0 = F(A_0)$ \\
$A_1$   &                                       \\
                 &                                         \\
                 & $B_1 = F(A_0,A_1)$ \\
$\vdots$         & $\vdots$ \\
\end{tabular}
\hspace{1in}
\begin{tabular}{c|l}
\multicolumn{2}{c}{$\gfin^{\omega}(\Omega,\Omega)$}                        \\ \hline\hline
\multicolumn{2}{c}{}\\ 
ONE   & TWO ($F$)                                                \\ \hline
      &                                                     \\
$\mathcal{U}_0$ & $\mathcal{V}_0 = F(\mathcal{U}_0)$                                      \\
      &                                                     \\
      &                                                     \\
$\mathcal{U}_1$ & $\mathcal{V}_1=F(\mathcal{U}_0,\mathcal{U}_1)$                                    \\
      &                                                     \\
$\vdots$         & $\vdots$ \\
\end{tabular}
\end{center}

Let $A_0$ be ONE's first move in $\gfin^{\omega}(\Omega_{\mathbf 0},\Omega_{\mathbf 0})$, and enumerate $A_0$ bijectively as $(f_n:n<\omega)$. For each $n$, define 
\[
  U_n = \{x\in X: f_n(x) <\epsilon_0\}.
\]
Now $\mathcal{U}_0 = \{U_n:n<\omega\}$ is an element of $\Omega$, and thus a legitimate move of ONE in the game $\gfin^{\alpha}(\Omega,\Omega)$. Apply TWO's winning strategy $F$ to find $F(\mathcal{U}_0) = \{U_n:n\in S_0\}$ where $S_0$ is a finite subset of $\omega$, and define $\Phi(A_0) = \{f_n:n\in S_0\}$, which is a legitimate move of TWO in $\gfin^{\alpha}(\Omega_{\mathbf 0},\Omega_{\mathbf 0})$.

Consider an $\beta<\alpha$. Supposing that $A_{\beta}$ is ONE's move in inning $\beta$, enumerate $A_{\beta}$ bijectively as $(f_m:m<\omega)$. Then choose $\psi(\beta)$ to be the least $n<\omega$ such that $\gamma_n\le \beta<\gamma_{n+1}$ and define for each $m$:
\[
  U_m = \{x\in X:\, f_m(x) <\epsilon_{\psi(\beta)}\}.
\]
Define $\mathcal{U}_{\beta}:=\{U_m:m<\omega\}$,  Applying TWO's winning strategy $F$ in $\gfin^{\alpha}(\Omega,\Omega)$ we find a finite subset $S_{\beta}$ of $\omega$ such that  $F(\mathcal{U}_{\nu}:\nu\le \beta) = \{U_m:m\in S_{\beta}\}$. Finally we define
\[
  \Phi(A_{\nu}:\nu\le \beta) = \{f_m:m\in S_{\beta}\}.
\]
This defines a strategy $\Phi$ for TWO in the game $\gfin^{\alpha}(\Omega_{\mathbf 0},\Omega_{\mathbf 0})$. 

We now verify that $\Phi$ is a winning strategy for TWO. To this end consider a $\Phi$-play
\[
  A_0,\, \Phi(A_0),\cdots,\,A_{\beta},\Phi(A_{\nu}:\nu\le \beta), \cdots\,\,\,\beta<\alpha.
\]
We claim that this play is won by TWO. 

Consider the associated $F$-play
\[
  \mathcal{U}_0,\, F(\mathcal{U}_0),\, \cdots,\, \mathcal{U}_{\beta},\, F(\mathcal{U}_{\nu}:\nu\le \beta),\, \cdots\,\,\, \beta<\alpha
\]
where for each $\beta$ we have
\begin{itemize}
\item[(a)]{$\mathcal{U}_\beta = \{\{x\in X:f_m(x)<\epsilon_{\psi(\beta)}\}: m<\omega\}$ and $U^{\beta}_m= \{x\in X: f_m(x)<\epsilon_{\psi(\beta)}\}$}
\item[(b)]{$S_{\beta}\subset \omega$ is a finite set such that $F(\mathcal{U}_{\nu}:\nu\le\beta) = \{U^{\beta}_m:m\in S_{\beta}\}$.}
\item[(c)]{$\Phi(A_{\nu}:\nu\le\beta) = \{f_m:m\in S_{\beta}\}$.}
\end{itemize}

Since this is a winning play for TWO in the game $\gfin^{\alpha}(\Omega,\Omega)$, $\bigcup \{F(\mathcal{U}_{\nu}:\nu\le\beta):\beta<\alpha\}$ is an $\omega$-cover of $X$, and sinceTWO's winning strategy has the property of Proposition \ref{typofstrfortwo}, for each $\gamma<\alpha$ the set $\bigcup\{F(\mathcal{U}_{\nu}:\nu\le \beta): \gamma < \beta<\alpha\}$ is an $\omega$-cover of $X$.

Consider a basic neighborhood of ${\mathbf 0}$, say $\{f\in {\sf C}_p(X): (\forall x\in H)(\vert f(x)\vert<\epsilon\}$ where $H\subset X$ is some finite set and $\epsilon$ is some positive real number. Choose $m$ large enough that $\epsilon_m< \epsilon$. Since $\bigcup \{F(\mathcal{U}_{\nu}:\nu\le \beta): m\le \psi(\beta) \mbox{ and } \beta<\alpha\}$ is an $\omega$-cover of $X$, choose a $\beta>\gamma_m$ and a $U^{\beta}_k \in F(\mathcal{U}_{\nu}:\nu\le\beta)$ with $H\subseteq U^{\beta}_k$. Now $U^{\beta}_k$ is of the form  $\{x\in X:\vert f_k(x) <\epsilon_{\psi(\beta)}\}$. But then as $f_k$ is a non-negative function, for each $x\in U^{\beta}_k$ we also have $f_k(x)<\epsilon_{\psi(\beta)}<\epsilon$, and thus $f_k$, a member of TWO's move in $\gfin^{\alpha}(\Omega_{\mathbf 0},\Omega_{\mathbf 0})$, is a member of the given neighborhood $\{f\in {\sf C}_p(X): (\forall x\in H)(\vert f(x)\vert<\epsilon\}$ of ${\mathbf 0}$.

It follows that TWO won the corresponding $\Phi$-play of $\gfin^{\alpha}(\Omega_{\mathbf 0},\Omega_{\mathbf 0})$.
\end{proof}

Theorem \ref{twowinsduality} has the following corollaries:

\begin{corollary}\label{ordinalvalues}
If $X$ is a ${\sf T}_{3.5}$ space, then 
\[
 {\sf tp}_{fin-\omega}(X) = {\sf tp}_{fin-ft}({\sf C}_p(X))
\]
\end{corollary}

This in turn gives the following generalization of Bella's result for ${\sf T}_{3.5}$-spaces in which each closed set is a ${\sf G}_{\delta}$-set:
\begin{corollary}\label{bellageneralize}
If $X$ is a ${\sf T}_{3.5}$-space, then TWO has a winning strategy in $\gfin^{\omega}(\Omega,\Omega)$ on $X$ if, and only if, TWO has a winning strategy in $\gfin^{\omega}(\Omega_{\mathbf 0},\Omega_{\mathbf 0})$ on ${\sf C}_p(X)$.
\end{corollary}

The corresponding theorem for the games $\gone^{\alpha}(\Omega,\Omega)$ and $\gone^{\alpha}(\Omega_{\mathbf 0},\Omega_{\mathbf 0})$ is as follows:

\begin{theorem}\label{twowinsdualityg1}
Let $X$ be a ${\sf T}_{3.5}$ space and let $\alpha$ be an ordinal.  Then the following are equivalent: 
   \begin{enumerate}
   \item{TWO has a winning strategy in the game $\gone^{\alpha}(\Omega_{\mathbf 0},\Omega_{\mathbf 0})$ on ${\sf C}_p(X)$.}
   \item{TWO has a winning strategy in the game $\gone^{\alpha}(\Omega,\Omega)$ on $X$.}
   \end{enumerate}
\end{theorem}
\begin{proof}
This equivalence was proven for $X$ a set of real numbers  in Theorem 9 of \cite{Gamelength}. The proof of (2) $\Rightarrow$ (1) proceeds as in the proof of Theorem 9 (1) of \cite{Gamelength}. For the proof of (1) $\Rightarrow$ (2) we proceed \emph{mutatis mutandis} as in proof of the corresponding implication in Theorem \ref{twowinsduality}.
\end{proof}
Theorem \ref{twowinsdualityg1} generalizes the corresponding part of Theorem 9 of \cite{Gamelength}, where it was proven if $X$ is a separable metric space.
\section{Selective versions of separability}

For a topological space $X$ we already introduced the notation $ \mathfrak{D} = \{D\subseteq X: \, D \mbox{ is dense in }X\}$ in Section 3. The selection principles $\sone(\mathfrak{D},\mathfrak{D})$ and $\sfin(\mathfrak{D},\mathfrak{D})$ were introduced in \cite{COC6}, and have been extensively studied since. For a topological space each of these selection properties implies that each dense subset of the space has a countable subset that is dense in the space. In particular this implies that the space has only countably many isolated points. In this section we make the assumption that the considered spaces are at least ${\sf T}_3$ and that each dense subset has a countable subset still dense in the space.

It was shown in Proposition 19 of \cite{COC6} that HFD's satisfy $\sone(\mathfrak{D},\mathfrak{D})$. This prior result follows from Theorem \ref{twowinshfdtightness} on account of the following\footnote{This result is well-known. A proof is included for the reader's convenience.}: 

\begin{lemma}\label{fantightsep} If $X$ has countable (strong) fan tightness at each of its elements, and if $X$ is separable, then $X$ has the properties $\sfin(\mathfrak{D},\mathfrak{D})$ (respectively, $\sone(\mathfrak{D},\mathfrak{D})$).
\end{lemma}
$\Proof$ Let $D$ be a fixed countable dense subset of $X$, enumerated bijectively as $(d_n:n<\omega)$. Let $(D_n:n<\omega)$ be a sequence of dense subsets of $X$. Write $\omega=\bigcup_{m\in\omega}S_m$ where each $S_m$ is infinite and for $m<n$, $S_m\cap S_n=\emptyset$. Now for each $m$ we have $x_m\in\overline{D}_n$, $n\in S_m$. If $X$ has countable strong fan tightness at each of its elements, choose for $n\in S_m$ an $x_n\in D_n$ such that $x\in\overline{\{x_n:n\in S_m\}}$, using countable strong fan tightness at each point. But then $\{x_n:n\in\omega\}$ is dense, witnessing $\sone(\mathfrak{D},\mathfrak{D})$. The argument for countable fan tightness is similar. $\Box$

\begin{corollary}\label{septight} If $X$ is a separable space which is countably tight at each element and if $\kappa$ is an uncountable cardinal, then
\[
  \mathbf{1}_{\cohen(\kappa)}\forces ``\check{X} \mbox{ satisfies } \sone(\mathfrak{D},\mathfrak{D})".
\]
\end{corollary}
\begin{proof} Use Theorem \ref{twowinshfdtightness}, Lemma \ref{fantightsep}, and the fact that separability is preserved by any forcing.
\end{proof}

\begin{corollary}\label{heredsep} If $X$ is a hereditarily separable space and $\kappa$ is an uncountable cardinal, then
\[
  \mathbf{1}_{\cohen(\kappa)}\forces ``\check{X} \mbox{ satisfies } \sone(\mathfrak{D},\mathfrak{D})".
\]
\end{corollary}
\begin{proof} By Lemma \ref{HSistight} a hereditarily separable space has countable tightness at each of its elements. Use Corollary  \ref{septight}.
\end{proof}

By the technique used in the proof of Theorem \ref{cohenrealsconvertcountablytight} each of these corollaries can be strengthened to:
\begin{proposition}\label{onenotwinning} If $X$ is a separable space which is countably tight at each element and if $\kappa$ is an uncountable cardinal, then
\[
  \mathbf{1}_{\cohen(\kappa)}\forces ``\mbox{ONE has no winning strategy in }\gone^{\omega}(\mathfrak{D},\mathfrak{D}) \mbox{ on }\check{X}."
\]
\end{proposition}

In Theorem 13 of \cite{COC6} we established duality results for player ONE of $\gone^{\omega}(\mathfrak{D},\mathfrak{D})$ and player ONE of $\gone^{\omega}(\Omega,\Omega)$, and in Theorem 35 of \cite{COC6} we established the corresponding results for player ONE of $\gfin^{\omega}(\mathfrak{D},\mathfrak{D})$ and of $\gfin^{\omega}(\Omega,\Omega)$. 
For player TWO results are as follows: Extending Section 3 of \cite{Gamelength} we introduce the following ordinal function for a topological space $X$\footnote{Caution: In Section 3 of \cite{Gamelength} the symbol ${\sf tp}_d(X)$ denotes the least ordinal $\alpha$ such that TWO has a winning strategy in the game $\gone^{\alpha}(\mathfrak{D},\mathfrak{D})$ on ${\sf C}_p(X)$.}:
\[
  {\sf tp}_d(X) = \min\{\alpha:\,\mbox{ TWO has a winning strategy in the game }\gone^{\alpha}(\mathfrak{D},\mathfrak{D})\}.
\]
Note that ${\sf tp}_d(X)$ is well-defined since it is no larger than $\pi(X)$, the $\pi$-weight of $X$. Now also define 
\[
  {\sf tp}_{d-fin}(X) = \min\{\alpha:\,\mbox{ TWO has a winning strategy in the game }\gfin^{\alpha}(\mathfrak{D},\mathfrak{D})\}.
\]
It is evident that for a space $X$ we have ${\sf tp}_{d-fin}(X)\le {\sf tp}_{d}(X)$. In Lemma 10 of \cite{COC6} it was observed that if $X$ is a separable metric space, then each dense subset of ${\sf C}_p(X)$ has a countable subset that is still dense in ${\sf C}_p(X)$. In Theorem 9 of \cite{Gamelength} it was shown that
\begin{theorem}\label{metricsonetp} If $X$ is a separable metric space, then: If any of the ordinals ${\sf tp}_{\omega}(X)$, ${\sf tp}_{sf}({\sf C}_p(X))$ and ${\sf tp}_{d}({\sf C}_p(X))$ is countable, so are the others, and
\[
  {\sf tp}_{\omega}(X) = {\sf tp}_{sf}({\sf C}_p(X)) = {\sf tp}_{d}({\sf C}_p(X)).
\]
\end{theorem}
It follows from Theorem 4 of \cite{Gamelength} that for an infinite separable metric space $X$, ${\sf tp}_d({\sf C}_p(X))$ is an additively indecomposable ordinal.
It also follows for separable metric spaces $X$ that if any of these three ordinals is $\omega_1$, then so are the others. 
We now proceed to establish corresponding facts for the game $\gfin^{\alpha}(\mathfrak{D},\mathfrak{D})$.

\begin{theorem}\label{twodensitygame} Let $X$ be an infinite separable metric space and let $\alpha$ be an ordinal. The following are equivalent:
\begin{enumerate}
\item{TWO has a winning strategy in $\gfin^{\alpha}(\mathfrak{D},\mathfrak{D})$ on ${\sf C}_p(X)$.}
\item{TWO has a winning strategy in $\gfin^{\alpha}(\Omega,\Omega)$ on $X$.}
\end{enumerate}
\end{theorem}
\begin{proof}
The proof of (1)$\Rightarrow$(2) proceeds as in the proof of Theorem 9 (3) of \cite{Gamelength}, with the necessary changes made for $\gfin$ instead of $\gone$, as we now outline.
Assume that $\Phi$ is a winning strategy for TWO in the game $\gfin^{\alpha}(\mathfrak{D},\mathfrak{D})$ on ${\sf C}_p(X)$.  Let $D=\{d_n:n<\omega\}$ be a countable dense subset of $X$, and for each $n$ and each $m$ let $B_{n,m}$ be the set $\{x\in X: d(x,d_n)<\frac{1}{3^{m+1}}\}$.  Then  $\mathcal{B} = \{B_{n,m}:\,m,\,n<\omega\}$ is a countable base for the topology of $X$.
Let $\mathcal{U}$ be the family of sets that are a union of finitely many distinct elements of $\mathcal{B}$ whose closures are disjoint. For each $U\in\mathcal{U}$, say 
\[
  U = \bigcup_{i=1}^{k} B_{n_i,m_i},
\]
and each finite set $\{q_1,\cdots,q_k\}$ (the same $k$ as in the definition of $U$) of rational numbers in the interval $(0,\,1)$, let $\chi_{U,q_1,\cdots,q_k}:X\rightarrow \lbrack 0,\,1\rbrack$ be the continuous function that is zero on $X\setminus U$, and takes value $q_i$ on $\overline{B}_{n_i,m_i+1}$.  

Since for an $\omega$-cover $\mathcal{V}$ of $X$ the set $\{U\in\mathcal{U}: (\exists V\in\mathcal{V})(U\subseteq V)\}$ is an $\omega$-cover of $X$, we may assume that all moves of ONE in the game $\gfin^{\alpha}(\Omega,\Omega)$ are subsets of $\mathcal{U}$. Also observe that for an $\omega$-cover $\mathcal{V}\subseteq\mathcal{U}$ the corresponding set $E(\mathcal{V})$ of functions of the form $\chi_{U,q_1,\cdots,q_k}$ with $U\in\mathcal{V}$ and $q_1,\cdots,q_k$ an appropriate set of rationals in $(0,\,1)$ is a dense subset of ${\sf C}_p(X)$ (See the proof of the Claim in the proof of \cite{Gamelength}, Theorem 9 part (3)).

Now we define a strategy $F$ for TWO in the game $\gfin^{\alpha}(\Omega,\Omega)$ on $X$. When ONE plays an $\omega$-cover $\mathcal{U}_0\subseteq\mathcal{U}$, TWO first computes the dense subset $E(\mathcal{U}_0)$ of ${\sf C}_p(X)$ and applying the winning strategy $\Phi$ computes the finite subset $\Phi(E(\mathcal{U}_0))$ of ${\sf C}_p(X)$. If this finite set is
\[
  \{\chi(U_1,q^1_1,\cdots,q^1_{k_1}), \cdots, \chi(U_m,q^m_1,\cdots,q^m_{k_m})\}
\]
then TWO responds with $F(\mathcal{U}_0) = \{U_1,\cdots,\,U_m\}$ in the game $\gfin^{\alpha}(\Omega,\Omega)$. 

Let $0<\beta<\alpha$ be given and suppose that in inning $\beta$ ONE plays the $\omega$-cover $\mathcal{U}_{\beta}\subseteq\mathcal{U}$. Then TWO first computes
\[
  \Phi(E(\mathcal{U}_{\nu}):\nu\le\beta) =   \{\chi(U_1,q^1_1,\cdots,q^1_{k_1}), \cdots, \chi(U_m,q^m_1,\cdots,q^m_{k_m})\},
\]
say, and then TWO responds with $F(\mathcal{U}_{\nu}:\nu\le\beta) = \{U_1,\cdots,U_m\}$.

We leave it to the reader to verify that $F$ is a winning strategy for TWO in the game $\gfin^{\alpha}(\Omega,\Omega)$.\\

{\flushleft{(2)$\Rightarrow$(1):}} We show that if TWO has a winning strategy in $\gfin^{\alpha}(\Omega_{\mathbf 0}.\Omega_{\mathbf 0})$, then TWO has a winning strategy in $\gfin^{\alpha}(\mathfrak{D},\mathfrak{D})$ in ${\sf C}_p(X)$, and thus as metrizable spaces are ${\sf T}_{3.5}$, Theorem \ref{twowinsduality} implies the result. We may assume that $\alpha$ is minimal. The proof now proceeds, \emph{mutatis mutandis}, like the proof of part (4) of Theorem 9 of \cite{Gamelength}.
\end{proof} 

We now expand the argument of Lemma \ref{fantightsep} to game lengths.

\begin{theorem}\label{fansepgame}
Let $X$ be a separable space.  Let $\alpha$ be an infinite ordinal.
\begin{enumerate}
\item{If for a dense set of $x\in X$ we have ${\sf tp}_{sf}(X,x)=\alpha$, then ${\sf tp}_d(X) \le \alpha$.}
\item{If for a dense set of $x\in X$ we have ${\sf tp}_{fin-ft}(X,x) = \alpha$, then ${\sf tp}_{d-fin}(X) \le \alpha$}
\end{enumerate}
\end{theorem}
\begin{proof}
We give the argument for part (1), leaving part (2) to the reader.
Let $D$ be a countable dense subset of $X$, enumerated as $(d_n:n<\omega)$. For each $n$, let $\sigma_n$ be a winning strategy for TWO in the game $\gone^{\alpha}(\Omega_{x_n},\Omega_{x_n})$. Since $\alpha$ is minimal, it is additively indecomposable. 

Write $\alpha = \bigcup\{S_n:n<\omega\}$ where for each $n<m$ both $S_n$ and $S_m$ are of order type $\alpha$, but pairwise disjoint. 
List each $S_n$ in order type $\alpha$ as $S_n = (\gamma^n_{\beta}:\beta<\alpha)$.

Define a strategy $\sigma$ for TWO in the game $\gone^{\alpha}(\mathfrak{D},\mathfrak{D})$ as follows:

Assuming inning $\gamma<\alpha$ is in progress and ONE has moved $(D_{\nu}:\nu\le \gamma)$, identify the $n$ with $\gamma\in S_n$. Then consider the sequence $(D_{\nu}:\nu\le \gamma \mbox{ and }\nu\in S_n)$. We define
\[
  \sigma(D_{\nu}:\nu\le \alpha) := \sigma_n(D_{\nu}:\nu\le\gamma\mbox{ and }\nu\in S_n).
\]
Then each $\alpha$-length $\sigma$-play of $\gone^{\alpha}(\mathfrak{D},\mathfrak{D})$ is won by TWO, proving that ${\sf tp}_d(X)\le\alpha$. 
\end{proof}

We have the following two characterizations of certain spaces where TWO has a winning strategy in the games of length $\omega$:
\begin{theorem}[\cite{COC6}, Theorem 3]\label{TWOgonedd} For a ${\sf T}_3$-space the following are equivalent:
\begin{enumerate}
\item{TWO has a winning strategy in $\gone^{\omega}(\mathfrak{D},\mathfrak{D})$.}
\item{$\pi(X)=\aleph_0$.}
\end{enumerate}
\end{theorem}

\begin{theorem}[Bella, \cite{ABPR} Corollary 3]\label{TWOgfindd} Let $X$ be a separable metric space. The following are equivalent:
\begin{enumerate}
\item{TWO has a winning strategy in $\gfin^{\omega}(\mathfrak{D},\mathfrak{D})$ on ${\sf C}_p(X)$.}
\item{$X$ is $\sigma$-compact.}
\end{enumerate}
\end{theorem}
It is of interest to find a characterization of spaces $X$ for which TWO has a winning strategy in the game $\gfin^{\alpha}(\mathfrak{D},\mathfrak{D})$ on $X$. 

\section{Remarks}

{\flushleft{\bf A}} T. Usuba further investigated the notion of indestructible countable tightness, and proved the following interesting theorem \cite{TU}:
\begin{theorem}[Usuba]\label{usubath} The following three theories are equiconsistent:
\begin{enumerate}
\item{{\sf ZFC} + {\sf every indestructibly countably tight space of cardinality $\aleph_1$ has character at most $\aleph_1$.}}
\item{{\sf ZFC} + {\sf $2^{\aleph_1}>\aleph_2$ and no countably tight space of cardinality $\aleph_1$ has character $\aleph_2$.}}
\item{{\sf ZFC} + {\sf there is a strongly inaccessible cardinal}.}
\end{enumerate}
\end{theorem}

{\flushleft{\bf B}} 
The countable tightness of a sequential space or of a Frech\'et-Urysohn space can also be destroyed by countably closed forcing.  
For a point $x$ of a space $X$ we say that the infinite subset $B$ of $X$ converges to $x$ if each for open neighborhood $U$ of $x$ the set $B\setminus U$ is finite. We define $\Gamma_x=\{B\subset X\setminus\{x\}: B \mbox{ converges to }x\}.$  It is clear that $\Gamma_x$ is a subset of $\Omega_x$.  A space is Frech\'et-Urysohn at $x$ if each element of $\Omega_x$ has a countable subset that is a member of $\Gamma_x$. A space is said to be \emph{strictly Frech\'et-Urysohn} at $x$ if it has the property $\sone(\Omega_x,\Gamma_x)$. When a space is Frech\'et-Urysohn at $x$ it has countable tightness at $x$. When it is strictly Frech\'et-Urysohn at $x$ then indeed it has countable strong fan tightness there. Also note that if ONE has no winning strategy in $\gone^{\omega}(\Omega_x,\Gamma_x)$, then ONE has no winning strategy in $\gone^{\omega_1}(\Omega_x,\Omega_x)$.  Thus, the following classical theorem of Sharma \cite{PS} implies that strictly Frech\'et-Urysohn spaces are indestructibly countably tight:
\begin{theorem}[Sharma \cite{PS}, Theorem 1]\label{sharmath}
Let $X$ be a topological space and let $x$ be an element of $X$. The following are equivalent:
\begin{enumerate}
\item{$X$ has property $\sone(\Omega_x,\Gamma_x)$.}
\item{ONE has no winning strategy in the game $\gone^{\omega}(\Omega_x,\Gamma_x)$.}
\end{enumerate}
\end{theorem}

Since for $X$ a ${\sf T}_{3.5}$-space, Gerlits and Nagy \cite{GN} proved that ${\sf C}_p(X)$ is Frechet-Urysohn at a point if, and only if, it is strictly Frechet-Urysohn there, it follows for $X$ a ${\sf T}_{3.5}$ space that if ${\sf C}_p(X)$ is Frech\'et-Urysohn, then it is indestructibly countably tight.

In a personal communication P.J. Szeptycki pointed out that countable tightness of a Frech\'et-Urysohn space can be destroyed by forcing with ${\sf Fn}(\omega_1,\omega,\omega_1)$, the countably closed partially ordered set for adding a Cohen subset of $\omega_1$. We discuss this in Example 8 below.

\section{Examples}

{\flushleft{\bf Example 1:}}  There are non-$\sigma$-compact ${\sf T}_{3.5}$ spaces $X$ of arbitrary large cardinality for which TWO has a winning strategy in the game $\gone^{\omega}(\Omega_{\mathbf 0},\Omega_{\mathbf 0})$ on ${\sf C}_p(X)$.

In Theorem 8 of \cite{RBRamsey} we gave for each infinite cardinal $\kappa$ an example of a ${\sf T}_{3.5}$ topological group $X$ which is a {\sf P}-space (each countable intersection of open sets is an open set), and for which TWO has a winning strategy in the game $\gone^{\omega}(\open,\open)$. Being a {\sf P}-space, $X$ is not embeddable as a closed subspace of a $\sigma$-compact space. By Theorem 1 of \cite{GN}, if for a ${\sf T}_{3.5}$-space TWO has a winning strategy in $\gone^{\omega}(\open,\open)$, then TWO has a winning strategy in $\gone^{\omega}(\Omega,\Gamma)$, and thus in $\gone^{\omega}(\Omega,\Omega)$. Then $X$ is not $\sigma$-compact, and yet by Theorem \ref{twowinsduality} TWO has a winning strategy in the game $\gone^{\omega}(\Omega_{\mathbf 0},\Omega_{\mathbf 0})$, and thus $\gfin^{\omega}(\Omega_{\mathbf 0},\Omega_{\mathbf 0})$, on ${\sf C}_p(X)$.

{\flushleft{\bf Example 2:}}  There are $\sigma$-compact ${\sf T}_{3.5}$ spaces $X$ of arbitrary large cardinality in which TWO has a winning strategy in the game $\gone^{\omega}(\Omega_{\mathbf 0},\Omega_{\mathbf 0})$ on ${\sf C}_p(X)$.

In Theorem 14 of \cite{RBRamsey} we gave for each infinite cardinal $\kappa$ an example of a ${\sf T}_{3.5}$ topological group $X$ of cardinality $\kappa$ which is $\sigma$-compact, and for which TWO has a winning strategy in the game $\gone^{\omega}(\Omega,\Omega)$. By Theorem \ref{twowinsduality} TWO has a winning strategy in the game $\gone^{\omega}(\Omega_{\mathbf 0},\Omega_{\mathbf 0})$, and thus $\gfin^{\omega}(\Omega_{\mathbf 0},\Omega_{\mathbf 0})$, on ${\sf C}_p(X)$.

{\flushleft{\bf Example 3:}}  Non $\sigma$-compact metric spaces $X$ for which TWO has a winning strategy in $\gfin^{\omega^2}(\Omega_{\mathbf 0},\Omega_{\mathbf 0})$ in ${\sf C}_p(X)$.

In Theorem 10 of \cite{Gamelength} {\sf CH} is used to construct an uncountable set $X$ of real numbers such that TWO has a winning strategy in $\gone^{\omega^2}(\Omega,\Omega)$ on $X$. It follows that $X$ has strong measure zero and thus, as $X$ is uncountable, it is not $\sigma$-compact. It follows that TWO has a winning strategy in $\gfin^{\omega^2}(\Omega,\Omega)$ on $X$, but not in $\gfin^{\omega}(\Omega,\Omega)$ (by Theorem \ref{twodensitygame} and Theorem \ref{TWOgfindd} as $X$ is not $\sigma$-compact). But then by Theorem \ref{twowinsduality} TWO has a winning strategy in $\gfin^{\omega^2}(\Omega_{\mathbf 0},\Omega_{\mathbf 0})$ on ${\sf C}_p(X)$, and by Proposition \ref{finaddindec} $\omega^2$ is the minimal such ordinal for this example.

{\flushleft{\bf Example 4:}} A space $X$ which is not indestructibly countably tight at an element $x$, but TWO has a winning strategy in $\gfin^{\omega}(\Omega_x,\Omega_x)$. 

It is well known that $\,^{\omega_1}2$ is destructibly Lindel\"of. Since it is compact in all finite powers, it is Lindel\"of in all finite powers and so by the Arkhangel'skii-Pytkeev Theorem, ${\sf C}_p(^{\omega_1}2)$ has countable tightness. As $\,^{\omega_1}2$ is a compact ${\sf T}_{3.5}$ space, Theorem \ref{twowinsduality} implies that TWO has a winning strategy in the game $\gfin^{\omega}(\Omega_{\mathbf 0},\Omega_{\mathbf 0})$ on ${\sf C}_p(^{\omega_1}2)$. 

We shall now see that the countable tightness of this space is destructible by countably closed forcing: We force with $\poset = {\sf Fn}(\omega_1,2,\omega_1)$. In the ground model define for each $\alpha<\omega_1$ and $i\in\{0,\,1\}$ the open set $U^{\alpha}_i = \{f\in \,^{\omega_1}2:\,f(\alpha)=i\}$. Let $(^{\omega_1}2)_G$ denote the ground model version of $^{\omega_1}2$. If $f$ is $\poset$-generic, then 
\[
  \mathcal{U} = \{U^{\alpha}_{f(\alpha)}:\alpha<\omega_1\}
\]
is an open cover of $(^{\omega_1}2)_G$, and has no countable subset that covers $(^{\omega_1}2)_G$. Let $\mathcal{V}$ be the set of finite unions of elements of $\mathcal{U}$. Then $\mathcal{V}$ is an open $\omega$-cover of $(^{\omega_1}2)_G$, but it has no countable subset which covers $(^{\omega_1}2)_G$. Now each $V\in\mathcal{V}$ is a ground model open set and thus for each $x\in \,(^{\omega_1}2)_G\setminus \overline{V}$ there is a ground model continuous function $f_U$ such that $f_U\lbrack\overline{U}\rbrack \subseteq \{0\}$ and $f_U(x) = 1$. The set 
\[
 \{f_U:U\in\mathcal{V}\}
\]
(in the generic extension) is a subset of the ground model version of the set ${\sf C}_p(^{\omega_1}2)$. In the generic extension this uncountable set has the zero function in its closure, but no countable subset of it does since $f$ is generic. Thus the ground model version of the set ${\sf C}_p(^{\omega_1}2)$, which had countable fan tightness in the ground model, is not countably tight in the extension.

This example also illustrates that although TWO has a winning strategy in $\gfin^{\omega}(\Omega_{\mathbf 0},\Omega_{\mathbf 0})$  on the space ${\sf C}_p(^{\omega_1}2)$, ONE has a winning strategy in the game $\gone^{\omega_1}(\Omega_{\mathbf 0},\Omega_{\mathbf 0})$. 
Thus, the game $\gfin^{\alpha}(\Omega_{\mathbf{0}},\Omega_{\mathbf{0}})$ is not as closely related to indestructibility of countable tightness.

{\flushleft{\bf Example 5:}} There is a countable space $X$ and point $x\in X$ at which $X$ has countable strong fan tightness, but ONE has a winning strategy in $\gfin^{\omega}(\Omega_x,\Omega_x)$, and thus in $\gone^{\omega}(\Omega_x,\Omega_x)$. This example is described on pages 250 - 251 of  \cite{COC3}.

{\flushleft{\bf Example 6:}} ${\sf ZFC} + \diamondsuit$: Consider Fedor\v{c}uk's compact, hereditarily separable, Hausdorff space which has no nontrivial convergent sequences. This space is not first countable, and thus by Theorem \ref{twowinscsf} TWO has no winning strategy in $\gone^{\omega}(\Omega_x,\Omega_x)$ at any $x\in X$. By Theorem \ref{weissth} TWO has a winning strategy in $\gone^{\omega_1}(\Omega_x,\Omega_x)$ at each $x\in X$, and thus $X$ is indestructibly countably tight. Moreover, as $X$ has countable $\pi$-weight by \cite{Fremlin} Corollary 43H, TWO has a winning strategy in the game $\gone^{\omega}(\mathfrak{D},\mathfrak{D})$ by Theorem \ref{TWOgonedd}. 

{\flushleft{\bf Example 7:}} A hereditarily separable ${\sf T}_2$ space for which there is no countable ordinal $\alpha$ such that TWO has a winning strategy in the game $\gone^{\alpha}(\Omega_x,\Omega_x)$. Player TWO does not have a winning strategy in the game $\gone^{\alpha}(\Omega,\Omega)$ on the real line $\reals$ for any countable ordinal $\alpha$. The space ${\sf C}_p(\reals)$ is hereditarily separable (see for example Theorem 1 of \cite{Park}), and thus TWO has a winning strategy in the game $\gone^{\omega_1}(\Omega_{\mathbf 0},\Omega_{\mathbf 0})$ by Theorem \ref{weissth}. But by Theorem \ref{metricsonetp} there is no countable ordinal $\alpha$ for which TWO has a winning strategy in the game $\gone^{\alpha}(\Omega_{\mathbf 0},\Omega_{\mathbf 0})$. Note also that this example illustrates that a hereditarily separable space need not have countable strong fan tightness.

{\flushleft{\bf Example 8:}} A Frech\'et-Urysohn space whose countable tightness is not indestructible. 
The underlying set of this space is $X = \{\infty\}\bigcup\{(\alpha,n): \, \alpha<\omega_1 \mbox{ and }n<\omega\}$. For a function $f:\omega_1\rightarrow\omega$ define
\[
 U_f = \{\infty\}\bigcup\{(\alpha,n):n>f(\alpha),\, \alpha<\omega_1,\, n<\omega\}.
\]
Then the topology of $X$ is defined so that each element of $\omega_1\times\omega$ is isolated, while neighborhoods of $\infty$ are sets of the form $U_f$ where $f:\omega_1\rightarrow \omega$ is a function.

Consider a subset $A$ of $X$ for which $\infty\in\overline{A}$. Then for each $f:\omega_1\rightarrow \omega$ we have $A\cap U_f$ is infinite. We claim that there is an $\alpha$ for which $S_{\alpha}:=\{n: (\alpha,n)\in A\}$ is infinite. For otherwise, define $g$ so that $g_A(\alpha) = 1 + \max\{n:(\alpha,\,n)\in A\}$. Then $U_{g_A}$ is a neighborhood of $\infty$ which is disjoint from $A$, a contradiction. But then $S_{\alpha}\subset A$ is a sequence in $X$ converging to $\infty$. It follows that $X$ has the Frech\'et-Urysohn property at $\infty$, and thus is countably tight at $\infty$.

We claim that ONE has a winning strategy in $\gone^{\omega_1}(\Omega_{\infty},\Omega_{infty})$. For $\alpha<\omega_1$ define $O_{\alpha} = \{(\beta,n):\alpha<\beta<\omega_1 \mbox{ and } n<\omega\}$. Then for each $\alpha$ we have $O_{\alpha}\in\Omega_{\infty}$. ONE's strategy is to originally play $O_0$. When TWO responds by choosing $(\alpha_1,n_1)$, ONE's response is $O_{\alpha_1+1}$, and so on. Observe that in any play where ONE follows this strategy, the moves made by TWO is a subset of some member $g$ of the set of functions from $\omega_1$ to $\omega$, and as such the set chosen by TWO during this play is not a member of $\Omega_{\infty}$. Thus, as ONE has a winning strategy in $\gone^{\omega_1}(\Omega_{\infty},\, \Omega_{\infty})$, whence by Theorem \ref{indestrtightchar} $X_{\infty}$ is not indestructibly countably tight at $\infty$.

\section*{Acknowledgements}
I thank Dr. P.J. Szeptycki for the communication that produced Example 8, and Dr. W.A.R. Weiss for permission to include his argument that TWO has a winning strategy in $\gone^{\omega_1}(\Omega_x,\Omega_x)$ whenever $x$ is an element of a hereditarily separable space (Theorem \ref{weissth}). I also thank Dr. T. Usuba for informing me of his pre-publication results \cite{TU} regarding indestructibility of countable tightness. Finally, I thank the referee for careful reading of the manuscript, and for several suggestions and remarks that greatly improved the paper.

\end{document}